\documentclass[10pt]{article}
\usepackage{amsmath, amssymb, amsthm, amscd,amsfonts,latexsym,hyperref}
\newtheorem{theorem}{Theorem}[section]
\newtheorem{lemma}[theorem]{Lemma}

\newtheorem{proposition}[theorem]{Proposition}
\theoremstyle{definition}
\newtheorem*{remarks}{Remarks}
\newtheorem*{remark}{Remark}
\newtheorem*{definition}{Definition}

\newcommand{\norm}[1]{\left|\left|#1\right|\right|}

\begin{document}

\title{Step Size in Stein's Method of Exchangeable Pairs}
%\date{}
\author{Nathan Ross \\ University of Southern California \\  Los Angeles CA 90089 \\ Email: nathanfr@usc.edu}
\maketitle

\begin{abstract} 
Stein's method of exchangeable pairs is examined through five examples in relation to Poisson and
normal distribution approximation.  In particular, in the case where the exchangeable pair is constructed from a reversible
Markov chain, we analyze how modifying the step size of the chain in a natural way affects the
error term in the 
approximation acquired through Stein's method.  It has been noted for the normal approximation that
smaller step sizes may yield better bounds, and we obtain the first rigorous results that verify this intuition. For the examples associated
to the normal distribution,
the bound on the error is expressed in terms of the spectrum of the underlying chain, a characteristic of the chain related to convergence rates.
The Poisson approximation
using exchangeable pairs is less studied than the normal, but in the examples presented here the same principles hold.   
\end{abstract}

\section{Introduction}

Stein's method has become a powerful tool in approximating probability
distributions and
proving central limit theorems. The various formulations of the method rely on exploiting the characterizing operator or ``Stein equation'' of the
distribution.
The characterizing
operator of a random variable $X$ is an operator $S$ such that, for a specified class of functions $A$, $\mathbb{E}Sf(Y)=0$ for all $f$ in $A$ if and only if 
$Y\buildrel d \over =X$.
Stein's
method can be used to quantitatively bound the difference between two random variables, one of which has a known characterizing operator.
In this paper we use theorems that obtain error terms from the characterizing operator through exchangeable pairs.
There are other variations of Stein's method that exploit the characterizing operator in different ways,
for example the zero bias transformation \cite{gol05, gore97}, the size bias coupling \cite{brs89, bhj92, gori96},
dependency graphs \cite{agg89, agg90}, and other ad hoc methods \cite{bol84, cha08, cha09, chsh01}.

For a gentle, intuitive explanation of Stein's method for normal
approximation see \cite{riro00}.  A more rigorous introduction can be found in \cite{stn86} and
similar ideas with more references in
\cite{rei98}. To find an introduction to Stein's method of exchangeable pairs for Poisson approximation
see \cite{cdm05}. For a very thorough introduction to Stein's method in general see \cite{bc05}.

An exchangeable pair is a pair of identically distributed random
variables
$(W,W')$ with the property that the distribution of $(W,W')$ is equal to the distribution of $(W',W)$.
The typical method used to generate useful exchangeable pairs on a finite space $\Omega$
is through a Markov chain $\{X_0,X_1,\ldots\}$ on $\Omega$ reversible
with respect to $\pi$.
%(i.e. $\pi(x)\mathbb{P}(X_{i+1}=y|X_i=x)=\pi(y)\mathbb{P}(X_{i+1}=x|X_i=y)$).
If $W$ is a random variable on $\Omega$, then it is easy
to show that setting
$W:=W(X_0)$ and $W':=W(X_1)$ where $X_0$ is chosen according to $\pi$ is an exchangeable pair. This is
not the only way to obtain an exchangeable pair: for exchangeable pairs from non-reversible 
Markov chains see \cite{ful04} and \cite{riro97}.

We will examine how modifying the step size
of the underlying Markov chain in a natural way affects the error term in the approximation acquired through Stein's method. 
In the case where the underlying Markov chain is ergodic, the step size does not necessarily affect the rate of convergence to stationarity
in a monotone way.
However, the rate of convergence is related to the eigenvalues of the Markov chain, and in the examples associated
to the normal distribution we are able to express the bound 
on the error in terms of the eigenvalues.

It will be
obvious in the sequel that modifying the exchangeable pair has a profound
effect on the error term.
Most notably, in the theorems we use, Markov chains with larger steps require more
computational work, and in the case of Poisson approximation, higher moment information.
For the examples presented here, the chains that allow for the easiest computation always yield the best bound. For other examples it is difficult to 
compute the error term for any chain other than the most computationally simple in a form that yields information about the relative sizes of the bounds.
Thus, it is difficult
to take these examples and make a rigorous statement about the step size of the underlying Markov chain and the bound acquired from it in
a general setting.

%Another lemma that will come in handy is for x an element of the state space,

%\begin{lemma}\label{jen}
%$Var(\mathbb{E}[(W'-W)^2|W]) \leq Var(\mathbb{E}[(W'-W)^2|x])$
%\end{lemma}

%Sometimes the right hand side is much easier to compute than the left and the direction of the inequality allows us to substitute
%the right hand side for the left in Theorem \ref{stn86}.

Section \ref{NA} introduces Stein's method of exchangeable pairs for normal approximation and explains why the effect of step size on the
error term is not obvious.
Sections \ref{binsec}, \ref{krawsec}, and \ref{planchsec} each contain one example of Stein's method's approximation of respectively, the
binomial (with $p=1/2$), Plancherel measure of the Hamming scheme, and Plancherel measure of the irreducible representations of a group G
by the normal
distribution.
Section \ref{lamsec} is tangent to Sections \ref{binsec}, \ref{krawsec}, and \ref{planchsec} in that the bounds on the error from
those sections are restated
in terms of the eigenvalues of the chain.  The
final two sections examine the approximation of the binomial and the negative binomial by the Poisson
distribution.

This is a small step in examining and refining Stein's method to be more widely applicable.  This paper serves to illustrate
the type of computations needed to apply the method and will hopefully yield some insight into the underlying theory behind Stein's method.
%\begin{definition}

%Two real valued random variables $(W,W')$ on a probability space $(\Omega, \cal{F}, \mathbb{P})$
%are said to be exchangeable if

 %\begin{equation}
%\mathbb{P}(W \in A, W' \in B) =  \mathbb{P} (W \in B, W' \in A)  
%\label{exchangeable}
%\end{equation}
%for all A,B measurable subsets of $\mathbb{R}$.

%\end{definition}

%\begin{proposition}

%Let $\pi$ the stationary distribution of a reversible Markov chain $\{X_0,X_1,\ldots \}$ with a
%finite
%state space $\Omega$.   
%Then if W is a
%random variable on $\Omega$ and $\mathbb{P}(X_0=x)=\pi(x)$, then $W=W(X_0)$ and $W'=W(X_1)$ is an
%exchangeable pair.

%\end{proposition}

%\begin{proof} Denote the probability that x goes to y in this Markov chain $P_{xy}$:

%\begin{align*}
%\mathbb{P}(X_0 =x,X_1=y) & =  \mathbb{P}(X_1=y|X_0=x) \,  \mathbb{P} (X_0=x) \\
 %& = P_{xy} \, \pi(x)  = P_{yx} \, \pi(y) \\
 %& = \mathbb{P}(X_0=y,X_1=x)
%\end{align*}

%\end{proof}

%%%%%%%%%%%%%%%%%%%%%%%%%%%%%%%%%%%%%%%%NEW SECTION %%%%%%%%%%%%%%%%%%%%%

%%%%%%%%%%%%%%%%%%%%%%%%%

       %%%%%%%%%%%%%%%%%%%%%%%%%%%%%%%%%%%%%%%%%%%%%%%%%%%%
       
\section{Normal Approximation}\label{NA}

For normal approximation, we use Stein's original theorem \cite{stn86}
not only because most other exchangeable pair formulations stem from it, but because in many situations
it still yields the best results.  Also, the theorem is stated in terms of the Kolmogorov metric, but the relative size of the bound on the error
is determined by the same terms in an analogous theorem where the Wasserstein metric is used \cite{bccs05}.
%Another similar theorem that is useful when $| W - W' |$
%is bounded is stated
%in \cite{riro97}.  

\begin{theorem} 
\cite{stn86}
\label{stn86}
Let $(W,W')$ an exchangeable pair of real random variables such that $\mathbb{E}(W'|W)=(1-a)W$
with $0<a<1$. Also, let $\mathbb{E}(W)=0$ and $\mathbb{E}(W^2)=1$. Then for all $x_0$ in $\mathbb{R}$,
\begin{align}
\bigg| \mathbb{P}(W<x_0)-\frac{1}{\sqrt{2\pi}}&\int_{-\infty}^{x_0}e^{\frac{-x^2}{2}}dx \bigg| \notag \\
&\leq 
\frac{\sqrt{Var(\mathbb{E}[(W'-W)^2|W])}}{a}
+(2\pi)^{-\frac{1}{4}}\sqrt{\frac{\mathbb{E}|W'-W|^3}{a}} \label{s861} \\
&\leq
\frac{\sqrt{Var(\mathbb{E}[(W'-W)^2|W])}}{a}+\left[\frac{\mathbb{E}(W'-W)^4}{\pi a}\right]^{1/4}. \label{s862}  
\end{align}
\end{theorem}

\begin{remarks}
\mbox{}
\begin{enumerate}
\item The first bound (\ref{s861}) is from \cite{stn86}, while (\ref{s862}) follows from the
Cauchy-Schwarz inequality and the straightforward computation $\mathbb{E}(W'-W)^2=2a Var(W)$ for an exchangeable pair $(W,W')$ with $\mathbb{E}(W'|W)=(1-a)W$. 

\item A result holds \cite{riro97} which contains error terms similar to Theorem \ref{stn86} and often yields better rates, but 
requires almost sure bounds on $|W - W'|$.

\item It has been shown \cite{roe08} that Theorem \ref{stn86} still holds without exchangeability
assuming instead that $W$ and $W'$ are equally distributed.  However, it is
unclear how useful this observation is in practice as the exchangeability plays a critical
role in defining the pair $(W,W')$ and in computing the error from Theorem \ref{stn86}.  
\end{enumerate}
\end{remarks}

It has been casually noted that
the error term from Theorem \ref{stn86}
should be small when $a$ is small, or equivalently when $W$ and $W'$ are ``close."
%However, when $a$ is small the numerators in (\ref{s861}) and (\ref{s862}) will be small.
%However, in this case $a$ will also be small, so that it is not clear how modifying the step size of
%the Markov chain will affect the error term.  
One line of thought \cite{riro97} that supports this idea is that if
$(W,W')$ is bivariate standard normal with covariance $(1-a)$, then $Var(\mathbb{E}[(W'-W)^2|W])=a^4Var(W^2)$.
This equation loosely implies that
smaller values of $a$ should yield better bounds if $(W,W')$ are nearly bivariate normal.
Another argument for the idea that the error term should be small when $W$ and $W'$ are close is that the bound from Theorem \ref{stn86}
follows from a Taylor series approximation of $W$ about $W'$. A simple heuristic to illustrate this argument can be found in
\cite{riro00} and \cite{rei98}.  However, the Taylor approximation takes place exclusively in the numerator of the error term and so
this argument does not take into account the denominator which also decreases when $W$ and $W'$ are close (the same observations can be
made directly from the error term).  In accordance with these remarks, we 
will take the value $a$ generated by the exchangeable pair $(W,W')$ of Theorem \ref{stn86} to be a rough
quantitative measure of the ``step size" as
referred to in the introduction.

Finally, we remark here that the families of chains from Sections \ref{krawsec} and \ref{planchsec} have a similar form. Both of these families
of chains are canonically engineered to satisfy the linearity
condition $\mathbb{E}[W'|W]=(1-a)W$.  To elaborate, if as described in the introduction, $(W,W')=(W(X_0),W(X_1))$
(that is, $\{X_0,X_1,\ldots\}$ is a stationary Markov chain reversible with respect to some distribution $\pi$ on a state space $\Omega$ and $W$ is some
function from $\Omega$ into $\mathbb{R}$),
then we have
\begin{align}
\mathbb{E}[W'|X_0=i]&=\sum_{j\in \Omega}\mathbb{P}(X_1=j|X_0=i)W(j)=\left(P\textbf{W}\right)_i, \notag
\end{align}
where $P$ is the transition matrix of the Markov chain and $\textbf{W}$ is the vector $W_i=W(i)$.  Roughly, the linearity condition implies
\[P\textbf{W}=(1-a)\textbf{W},\]
so that $\textbf{W}$ must be an eigenvector of the transition matrix $P$.

In both of Sections \ref{krawsec} and \ref{planchsec} the random variable $W$ under study is a member of an orthogonal (in the $L^2$ sense
with respect to the measure $\pi$ under study) family
of functions on the state space.  
From this point, it is not difficult to canonically define a matrix indexed by $\Omega\times\Omega$, with row sums equal to one,
which has the orthogonal family as eigenvectors,
and satisfies the detailed balance equations for the measure $\pi$
(although the entries are not guaranteed to be positive).  We omit a more detailed abstract formulation, but the main ideas can be found in
\cite{sen01}.

%As a couple of final remarks about Theorem \ref{stn86}, in some examples $\mathbb{E}\left[|W'-W|^3\right]$ is difficult to compute.  To obtain
%around this, we will sometimes use the well known Cauchy-Schwarz inequality for expected value.

%\begin{lemma}
%For X and Y random variables where the following expected values exist,

%\[\mathbb{E}[XY]\leq\sqrt{\mathbb{E}[X^2]\mathbb{E}[Y^2]}\]

%\end{lemma}

%Applying this to the second part of the error term from Theorem \ref{stn86} yields:

%\begin{lemma}\label{cs2}

%\[\mathbb{E}\left[|W'-W|^3\right]\leq\sqrt{\mathbb{E}\left[(W'-W)^2\right]\mathbb{E}\left[(W'-W)^4\right]}\]

%\end{lemma}

%One of the terms in the previous lemma always has the following expression:

%\begin{lemma}\label{CS}
%Given the hypotheses of Theorem \ref{stn86},

%\[\mathbb{E}\left[(W'-W)^2 \right]=2a\]

%\end{lemma}

%\begin{proof}
%\begin{align*}
%\mathbb{E}\left[(W'-W)^2 \right] & = \mathbb{E}\left(\mathbb{E}[(W'-W)^2|W]\right) \\
%								& =\mathbb{E}[(W')^2]-2\mathbb{E}\left(W\mathbb{E}[W'|W]\right) +\mathbb{E}[W^2] \\
%								& =2\left(\mathbb{E}[W^2]-\mathbb{E}[(1-a)W^2]\right) \\
%								& =2a
%\end{align*}
%Where the first equality is a well known property of conditional expectation, the third
%is by exchangeability and hypotheses from Theorem \ref{stn86}, and the fourth follows from $\mathbb{E}[W^2]=1$.
%\end{proof}

%%%%%%%%%%%%%%%%%%%%%%%%%%%%%%%%%%%%%%%%NEW SECTION%%%%%%%%%%%%%%%%%%%%%%%%%%%%%%%%%%%%%

\section{Binomial Distribution} \label{binsec}

    In this section, we examine the most basic example of Stein's method for normal approximation.  It is well known that
the binomial distribution with parameters $n$ and $p$ is approximately normal for large $n$.  Stein's method can be used to
obtain an error term in this approximation. The setup is as follows:  let $(X_1,X_2,\ldots,X_n)$ a random vector where
each $X_i$ is independent and equal to 1 with probability $p$ $(0<p<1)$ and $0$ with probability $(1-p)$.
Then take $X=\sum_{i=1}^nX_i$.  In order to clearly illustrate the typical way we will analyze the rest of the examples,
we will take $p=1/2$ in this section.

The first thing we need to study the error
term from Theorem \ref{stn86} is the family of Markov chains that induce the exchangeable pair.  Given a configuration
of an $n$ dimensional $0-1$ vector, the next step in the chain follows the rule of changing any fixed $i$ coordinates with probability
$b_i/\binom{n}{i}$, so that $\sum_{i=0}^nb_i=1$ ($b_i$ is the probability of moving to a configuration Hamming distance $i$ away).
Since the probability of going from a configuration
$x$ to a configuration $y$ is the same as going from $y$ to $x$, the random vector $(X_1, \ldots, X_n)$ with each coordinate an independent Bernoulli
random variable with parameter $p=1/2$
is clearly reversible with respect to these chains. Setting $X=\sum_{i=1}^nX_i$ and $X'=\sum_{i=1}^nX_i'$
where 
%$(X_1, \ldots, X_n)$ is as in the previous sentence and
$(X_1', \ldots, X_n')$ is a step in the chain
described above induces an exchangeable pair. In order to apply Theorem \ref{stn86}, the random variable must 
have mean 0 and variance 1. Set $W=\frac{X-\mathbb{E}(X)}{\sqrt{Var(X)}}=\frac{X-np}{\sqrt{np(1-p)}}$, and define $W'$
as $W$ but with $X'$ in place of $X$.  The final hypothesis from Theorem \ref{stn86} is the following lemma.

\begin{lemma}\label{bil}

\[\mathbb{E}(W'|W)=\left(1-2\sum_{i=1}^n\left(\frac{i}{n}\right)b_i\right)W.\]

\end{lemma}
\begin{proof}
Let $X=\sum_{i=1}^nX_i$ and $X'=\sum_{i=1}^nX_i'$ as defined above.

\begin{align*}
\mathbb{E}(X'-X|X)& =  \sum_{i=1}^n\mathbb{E}(X_i'-X_i|X)\\
 & = \sum_{i:X_i=1}(-1)\mathbb{P}(X_i'=0|X_i=1)+\sum_{i:X_i=0}\mathbb{P}(X_i'=1|X_i=0) \\
 & = (n-2X)\sum_{i=1}^n\left(\frac{i}{n}\right)b_i.
\end{align*}

Substituting $X=\sqrt{n/4}W+n/2$ and $X'=\sqrt{n/4}W'+n/2$ yields
$\mathbb{E}(W'-W|W)= -2\sum_{i=1}^n\left(\frac{i}{n}\right)b_iW$, which is the lemma. 
\end{proof}

Now Theorem \ref{stn86} can be applied with $a=2\sum_{i=1}^n\left(\frac{i}{n}\right)b_i$.  In order to apply the theorem,
we still need to compute the quantities
$Var(\mathbb{E}[(W'-W)^2|W])$ and $\mathbb{E}(W'-W)^4$.

\begin{lemma}\label{tm1}
\[Var(\mathbb{E}[(W'-W)^2|W]) = 16B^2Var(W^2),\] 
where
$B=\sum_{i=2}^n\frac{i(i-1)}{n(n-1)}b_i.$
\end{lemma}
\begin{proof}
\begin{align*}
\mathbb{E}[(X'-X)^2|X] & =  \sum_{i=1}^n\mathbb{E}[(X_i'-X_i)^2|X)  +   
							\mathop{\sum_{i:X_i=1}}_{j \not= i:X_j=0}\mathbb{E}[(X_i'-X_i)(X_j'-X_j)|X)  \\
						& +	\mathop{\sum_{i:X_i=1}}_{j \not= i:X_j=1}\mathbb{E}[(X_i'-X_i)(X_j'-X_j)|X)  \\
						& + \mathop{\sum_{i:X_i=0}}_{j \not= i:X_j=0}\mathbb{E}[(X_i'-X_i)(X_j'-X_j)|X)  \\
						& = nA +(X(X-1)+(n-X)(n-X-1)-2X(n-X))B, \\
\end{align*}
where $A=\sum_{i=1}^n\left(\frac{i}{n}\right)b_i$.  Substituting $X=\sqrt{n/4}W+n/2$ and $X'=\sqrt{n/4}W'+n/2$ into the equation and solving appropriately yields
$\mathbb{E}[(W'-W)^2|W]=4BW^2+C$ where $C$ is some constant.
Taking variances proves the lemma.
\end{proof}

\begin{lemma}\label{tm2}
\[\mathbb{E}(W'-W)^4 =\frac{16}{n}(A+3B(n-1)+(2-3n+n\mathbb{E}(W^4))D,\]
where $A=\sum_{i=1}^n\left(\frac{i}{n}\right)b_i$, $B$ is defined as in Lemma \ref{tm1}, and $D$ is a constant. 
\end{lemma}

\begin{proof}
Let $Y_i=X_i'-X_i$.
\begin{align*}
\mathbb{E}\left[(X'-X)^4|X\right] &= \sum_{i=1}^n\mathbb{E}\left[Y_i^4|X\right]+ 
						4\mathop{\sum_{i}}_{j \not= i}\mathbb{E}\left[Y_i^3Y_j|X\right] \\  
			                      &+ 3\mathop{\sum_{i}}_{j \not= i}\mathbb{E}\left[Y_i^2Y_j^2|X\right] 
                        +6\mathop{\mathop{\sum_{i}}_{j \not= i}}_{l\not= i,j}\mathbb{E}\left[Y_i^2Y_jY_l|X\right] \\
		           &+\mathop{\mathop{\mathop{\sum_{i}}_{j \not= i}}_{l \not= i,j}}_{m \not= i,j,l}\mathbb{E}\left[Y_iY_jY_lY_m|X\right].\\
\end{align*}                                    
 
Counting the number of each type and computing the expected values
in a manner similar to the proof of Lemma \ref{tm1} yields

\begin{align*}
\mathbb{E}[(X'-X)^4&|X ]  
%= nA + 4[X(X-1)+(n-X)(n-X-1)  \\ 
%			&-2X(n-X)]B + 6(n-2)[X(X-1)\\
%		    &+(n-X)(n-X-1)-2X(n-X)]C + 3n(n-1)B \\
%			&+[n(n-1)(n-2)(n-3)-8((n-X)X(X-1)(X-2) \\
%			&+X(n-X)(n-X-1)(n-X-2))]D \\
            =nA + (4B+6(n-2)C)(n(n-1)-4X(n-X))\\
            &+[n(n-1)(n-2)(n-3)-8((n-X)X(X-1)(X-2) \\
			&+X(n-X)(n-X-1)(n-X-2))]D+3n(n-1)B.
\end{align*}
Here $C$ and $D$ are constants that vanish in the final expression, but $C$ is the probability of any fixed three
coordinates changing, and $D$ is the same probability but with four coordinates.
Substituting in $W$ and $W'$ as in the previous lemmas and taking the expected value of
$\mathbb{E}\left[(W'-W)^4|W\right]$ implies the lemma.
\end{proof}

\begin{lemma} \label{mom}
$\mathbb{E}[W^4]=3-\frac{2}{n}.$
\end{lemma}

\begin{proof}

The moment generating function for $W$ is

\vspace{2mm}

$\phi(t)=\left(\frac{1}{2}\left[exp  \left(\left(\frac{1}{n}\right)^{1/2}t\right)\right]+
\frac{1}{2}\left[exp  \left(-\left(\frac{1}{n}\right)^{1/2}t\right)\right]\right)^n,$

\vspace{2mm}

and $\mathbb{E}[W^4]=\phi^{(4)}(0)$.

\end{proof}

Now that we have all the formulas needed to apply Theorem \ref{stn86}, we can prove the main result of this section.

\begin{theorem}\label{binmthm}
Using the family of reversible Markov chains described previously, the error term given
by Theorem \ref{stn86} is minimum for $b_0+b_1=1$ (and $b_0 \not= 1)$.  In this case the bound is $\left[\frac{8}{\pi n}\right]^{1/4}$.
\end{theorem}

\begin{proof}

Because a is positive, it is sufficient to verify that
$\frac{Var(\mathbb{E}[(W'-W)^2|W])}{a^2}$ and
$\frac{\mathbb{E}(W'-W)^4}{a}$ are minimum for $b_0+b_1=1$.
First we examine $\frac{Var(\mathbb{E}[(W'-W)^2|W])}{a^2}$. By Lemmas \ref{bil}, \ref{tm1}, and \ref{mom} and the
fact that $a=2A$, this term is equal to
\[8\left(\frac{n-1}{n}\right)\left(\frac{B^2}{A^2}\right),\]
which is non-negative and equal to zero when $b_0+b_1=1$.

For the second term we use the second formulation in Theorem \ref{stn86} and find the minimum of
\begin{align}
\frac{\mathbb{E}(W'-W)^4}{a}. \label{bi4}
\end{align}
By Lemmas \ref{bil}, \ref{tm2}, and \ref{mom}, (\ref{bi4}) is equal to
\[\frac{8}{n}\left(1+3(n-1)\frac{B}{A}\right),\]
which is also minimized when $b_0+b_1=1$, and is equal to $8/n$ in this case.
\end{proof}

\begin{remarks}
\mbox{}
\begin{enumerate}
\item It is important to notice that the error term from Theorem \ref{stn86} depends on the Markov chain only through the term $B/A$, where
$B=\sum_{i=1}^n\left(\frac{i(i-1)}{n(n-1)}\right)b_i$ and $A=\sum_{i=1}^n\left(\frac{i}{n}\right)b_i$. We will be using this fact later
in Section \ref{lamsec}.

\item The case $b_0+b_1=1$ corresponds to the Markov chain that holds with probability $b_0$ and changes one coordinate chosen uniformly at
random with probability $b_1$. Therefore, this chain has the smallest maximum step size over all the chains in the family under study. Also,
as mentioned in the previous section, a quantitative measure of step size is the associated value of $a$ from Theorem \ref{stn86}, which is
equal to $2b_1/n$ in this optimal case.  Restricting to the case where $b_0=0$, the chain yielding the best bound ($b_1=1$) has the smallest
value of $a$.  This restriction is not artificial; in lieu of the previous remark,
the chain generated with $\mathbb{P}(T=t)=b_t'$, where $b_0'=0$, and $b_i'=b_i/(1-b_0)$ for $i\not=0$ has the same
error term from Theorem \ref{stn86} as the chain with $\mathbb{P}(T=t)=b_t$.  In other words, manipulating the holding probability
changes the parameter $a$, but can not yield bounds better than those obtained by choosing $b_0=0$. 
\end{enumerate}
\end{remarks}

%%%%%%%%%%%%%%%%%%%%%%%%%%%%%%%%% NEW SECTION %%%%%%%%%%%%%%%%%%%%%%%%%%%%%%%%%%%%%%%%%%%%%%%%%%%%%%%%
%%%%%%%%%%%%%%%%%%%%%%%%%%%%%%%%%%%%%%%%%%%%%%%%%%%%%%%%%%%%%%%%%%

\section{Plancherel Distribution of the Hamming Scheme (Or the Binomial Distribution in General)} \label{krawsec}

In this section we examine the uniform distribution on the eigenvalues of the adjacency matrix
for the Hamming graph. The Hamming graph $H(n,q)$ has vertex set $X$ equal to n-tuples of $\{1,2, \ldots ,q\}$
(thus $|X|=q^n$) with an edge between two vertices if they differ in exactly one coordinate.  

The following information about the adjacency matrix of the hamming scheme can be found in \cite{bait84} in the more generalized setting
of association schemes.
Let $v_i=(q-1)^i \binom{n}{i}$, the number of vertices that differ from a fixed vertex by
$i$ coordinates.
The eigenvalues of $H(n,q)$ are $K_1(i)=n(q-1)-qi$ with multiplicity $v_i$
for $i=0,1,2, \ldots ,n$.   Choose i with probability $\frac{v_i}{|X|}$ and designate this the Plancherel distribution of the Hamming Scheme.
Let $W(i)=\frac{K_1(i)}{\sqrt{v_1}}$ a random variable with unit variance.  In order to define the
family of Markov chains that induce the exchangeable pairs, we must define the q-Krawtchouk polynomials:
\[K_j(i)=\sum_{l=0}^j (-1)^l(q-1)^{j-l}\binom{i}{l}\binom{n-i}{j-l}.\]
Here and in what follows we freely use the convention $\binom{m}{r}=0$ for $r>m$ or $r<0$.
Following \cite{ful05},
define
\begin{align}
L_T(i,j)=\frac{v_j}{|X|}\sum_{r=0}^n \frac{K_r(i)K_r(T)K_r(j)}{v_r^2}. \label{LTdef}
\end{align}
For a given $T$ in $\{1, \ldots,n\}$, define a Markov chain on $\{0,\ldots,n\}$ by the transition probability of moving from $i$ to $j$ as $L_T(i,j)$. Then $L_T(i,j)$ is 
a Markov chain on $\{0,\ldots,n\}$ reversible with respect to the Plancherel distribution above $\left(\mathbb{P}(i)=\frac{v_i}{|X|}\right)$.
Following the usual setup, choose i from the Plancherel distribution and then j with probability $L_T(i,j)$ and set
the exchangeable pair $(W,W')=(W(i),W(j))$. In \cite{ful05}, the author uses Theorem \ref{stn86} with the exchangeable
pair induced by the chain $L_1$ (defined by (\ref{LTdef}) with $T=1$) to obtain a bound on the difference between the normal distribution and
the Plancherel measure. We will show that over a large family of Markov chains, $L_1$ is the most local chain and 
the error term obtained in Theorem \ref{stn86} using $L_1$ (as was done in \cite{ful05}) is optimal over this family.

Another (somewhat more motivating) way of viewing the Hamming scheme is given in the following easily verified proposition. 
\begin{proposition}\cite{ful05}\label{bico}
For $W$ defined as above, $W$ is equal in distribution to a binomial distribution with parameters $n$ and $p=\frac{1}{q}$
%\[\left(\frac{Y-n\frac{1}{q}}{\sqrt{\frac{n(q-1)}{q^2}}}\right),\] 
normalized to have mean 0 and variance 1.
\end{proposition}

Thus the Plancherel distribution of the Hamming scheme can be defined as a binomial distribution.
We will redefine the Markov chain $L_T(i,j)$ in terms
of this characterization, but first we list some well known properties of
Krawtchouk polynomials found in \cite{masl77}.

\begin{lemma}\label{kraw01}
For $j,l \in \{0,\ldots,n\}$, 
\[\sum_{i=0}^n\frac{K_i(j)K_i(l)}{v_i}=\frac{|X|}{v_j}\delta_{j,l}.\]
\end{lemma}

\begin{lemma}\label{kraw2}
For $j,l \in \{0,\ldots,n\}$, 
\[\sum_{i=0}^nv_iK_j(i)K_l(i)=|X|v_j\delta_{j,l}.\]
\end{lemma}

\begin{lemma}\label{krawrec}
For $j \in \{0,\ldots,n\}$,
\[(j+1)K_{j+1}(i)=((n-j)(q-1)+j-qi)K_j(i)-(q-1)(n-j+1)K_{j-1}(i).\]
\end{lemma}

\begin{lemma}\label{krawswap}
For $i,j \in \{0,\ldots,n\},$
\[K_j(i)=(q-1)^{j-i}\frac{\binom{n}{j}}{\binom{n}{i}}K_i(j).\]
\end{lemma}

Finally, we need one more tool that equates the product of two Krawtchouk polynomials with a linear
combination of Krawtchouk polynomials.  
\begin{lemma}\label{krawbase}
For $i, j, r  \in \{0,\ldots,n\},$
\[K_i(r)K_j(r)=    \sum_{l=-j}^jA_{j,i+l}(i)K_{i+l}(r),\]
where \[A_{j,i+l}(i)=\frac{\binom{n}{i}}{\binom{n}{i+l}}\sum_{k=0}^j\binom{j-k}{k-l}\binom{n-i}{k}\binom{i}{j-k}\frac{(q-2)^{j-2k+l}}{(q-1)^{l-k}}.\]

\end{lemma}
\begin{proof}
The first thing to note is that the Krawtchouk polynomials $K_i(r)$ for $i$ non-negative integers form a basis for all polynomials in $r$, so that
such a decomposition exists. Now, fix $i$ and write $A_{j,i+l}(i)=A_{j,i+l}$.
For $j=0$, $A_{0,i}=1$ and $A_{0,i+l}=0$ for $l \not = 0$, which agrees with the lemma.

Also, since $K_1(r)=(q-1)(n-r)-r=(q-1)n-qr$, Lemma \ref{krawrec} implies
\begin{align}
K_i(r)K_1(r)=i(q-2)K_i(r)+(i+1)K_{i+1}(r)+(q-1)(n-i+1)K_{i-1}(r), \label{krpr1}
\end{align}
which is also consistent with the lemma (the equality (\ref{krpr1}) was shown in \cite{ful05}).

For $j \geq 1$, we use Lemma \ref{krawrec} and strong induction to obtain
\begin{align}
(j+1)K_{j+1}&(r)K_i(r)=\notag\\ 
&((n-j)(q-1)+j-qr)K_j(r)K_i(r) \notag \\
&-(q-1)(n-j+1)K_{j-1}(r)K_i(r) \notag  \\
                     &= \left((n-j)(q-1)+j-qr\right)\sum_{l=-j}^jA_{j,i+l}K_{i+l}(r) \notag \\
  &-(q-1)(n-j+1)\sum_{l=-(j-1)}^{j-1}A_{j-1,i+l}K_{i+l}(r)  \notag \\
  &=\left((n-j)(q-1)+j\right)\sum_{l=-j}^jA_{j,i+l}K_{i+l}(r) \notag \\
  &-(q-1)(n-j+1)\sum_{l=-(j-1)}^{j-1}A_{j-1,i+l}K_{i+l}(r) \label{kp3}\\
  &+\sum_{l=-j}^j[(l+i+1)K_{l+i+1}(r)+(q-1)(n-i-l+1)K_{i+l-1}(r) \notag \\
 &-((n-i-l)(q-1)+l+i)K_{i+l}(r)]A_{j,i+l}, \notag 
\end{align}
where the final equality is by Lemma \ref{krawrec}. For each $l$, the coefficient
of $K_{i+l}(r)$ in (\ref{kp3}) is
\begin{align}
&(i+l-j)(q-2)A_{j,i+l}+(q-1)(n-i-l)A_{j,i+(l+1)} \label{kp1}\\
&+(i+l)A_{j,i+(l-1)}-(q-1)(n-j+1)A_{j-1,i+l}\label{kp2}.
\end{align}
The lemma will follow if the expression above is equal to $(j+1)A_{j+1,i+l}$.
To see this fact, re-index the sums in the definition of
$A_{j,i}$ in (\ref{kp2}) to begin at one, and equate summands.
\end{proof}

The fact from the previous lemma that the coefficients in the linear expansion
are positive for $q\geq 2$ has been shown without explicit computation in \cite{eag69} and restated in the monograph \cite{ask75}. 
We use that the coefficients are positive in the next theorem which shows they can be used to define a probability distribution.

\begin{theorem} \label{LT}
For $q\geq2$, the Markov chain $L_T(i,j)$ defined on $\{0,\ldots,n\}$ has the same
transition probabilities as the following chain:  Given a $0-1$ n-tuple
with $i$ ones, choose $T$ coordinates at random.  Replace every zero coordinate chosen to a one and for each one coordinate chosen, replace it with a 
zero with probability $\frac{1}{q-1}$ and let the coordinate remain as a one with probability $\frac{q-2}{q-1}$.
The probability of going from $i$ ones to $j$
ones is $L_T(i,j)$. 
\end{theorem}

\begin{proof}
By summing over $k$ equal to the number of zeros in the $T$ coordinates chosen, 
the probability of going from $i$ to $i+l$ in the chain described is
\[\sum_{k=0}^T\binom{T-k}{k-l}\left(\frac{(q-2)^{T-2k+l}}{(q-1)^{T-k}}\right)\left(\frac{\binom{n-i}{k}\binom{i}{T-k}}{\binom{n}{T}}\right).\]

Also,
\begin{align*}
L_T(i,j)&=\frac{v_j}{|X|}\sum_{r=0}^n \frac{K_r(i)K_r(T)K_r(j)}{v_r^2}\\
		&=\frac{v_j}{|X|\binom{n}{i}\binom{n}{j}\binom{n}{T}(q-1)^{i+j+T}}\sum_{r=0}^n(q-1)^r\binom{n}{r}K_i(r)K_T(r)K_{j}(r)\\
		&=\frac{v_j}{|X|\binom{n}{i}\binom{n}{j}\binom{n}{T}(q-1)^{i+j+T}}\sum_{r=0}^nv_rK_j(r)\sum_{l=-T}^TA_{T,i+l}K_{i+l}(r)\\
		&=\frac{v_j}{|X|\binom{n}{i}\binom{n}{j}\binom{n}{T}(q-1)^{i+j+T}}\sum_{l=-T}^TA_{T,i+l}\sum_{r=0}^nv_rK_{i+l}(r)K_{j}(r).\\
\end{align*}
The first equality is by Lemma \ref{krawswap} and the second is by Lemma \ref{krawbase}.  

Applying Lemma \ref{kraw2} implies \[L_T(i,i+l)=\frac{v_{i+l}}{\binom{n}{i}\binom{n}{T}(q-1)^{i+T}}A_{T,i+l}.\]
By Lemma \ref{krawbase}, this equals the desired quantity.
\end{proof}

According to Proposition \ref{bico}, the restriction $q\geq2$ corresponds to a binomial distribution with parameters $n$ and $p$ with $1/2\leq p <1$.
However, after normalizing to have mean equal to zero and unit variance, a binomial random variable with parameters $n$ and $(1-p)$ is the negative
of a binomial random variable with parameters $n$ and $p$.  Therefore, because the normal distribution is symmetric about zero, the following analysis 
can be applied to any binomial random variable.

The chains defined by ($\ref{LTdef}$) have now been fully described in terms of the binomial distribution.
Now we will start to examine the error term from Theorem \ref{stn86}. The next lemma shows the quantity $a$ from Theorem \ref{stn86} is 
equal to $\frac{qT}{n(q-1)}$
so that $a$ is increasing as a function of $T$.  Also, Theorem \ref{LT} implies that the maximum step size of $L_T$ is $T$; smaller
values of $T$ make $W$ and $W'$ ``closer" in both senses described in Section \ref{NA}. 

\begin{lemma}\label{krawlam}\cite{ful05}
$\mathbb{E}[W'|W]=\left(\frac{K_1(T)}{v_1} \right)W.$
\end{lemma}

\begin{proof}
\begin{align*}
\mathbb{E}[W'|i] &=\frac{1}{\sqrt{v_1}}\sum_{j=0}^nL_T(i,j)K_1(j) \\
				&=\frac{1}{\sqrt{v_1} \, |X|}\sum_{r=0}^n\frac{K_r(T)K_r(i)}{v_r^2}\sum_{j=0}^nv_jK_r(j)K_1(j) \\
				&=\left(\frac{K_1(T)}{v_1} \right)W(i).
\end{align*}
The first equality is by definition and the second uses Lemma \ref{kraw2} directly.  Since conditioning on $i$
only depends on $i$ through $W$, the lemma is proven.
\end{proof}

For this example, the error term (in the more general setting of association schemes) has already been computed using Theorem \ref{stn86}
in \cite{ful05}, so we only state the result we need.  First we must define a function $p_{ \, 2}$ on $\{0,\ldots,n\}$ using the following description.
Start from a fixed $X_0$ in $X$ and choose a coordinate uniformly at random.
Replace the chosen coordinate by one of the remaining $q-1$ options different from the original value uniformly at random to obtain $X_1$.  Perform
the same operation on $X_1$ to obtain $X_2$ and  
define $p_2(j)$ to be the probability that $X_2$ has $j$ coordinates
different then $X_0$.  Thus $p_2(0)=\frac{1}{n(q-1)}$, $p_2(1)=\frac{(q-2)}{n(q-1)}$, $p_2(2)=\frac{n-1}{n}$, and $p_2(j)=0$ for $j \geq 3$.

\begin{theorem}\cite{ful05}\label{krawful05}
Let $W$ and $a$ defined as above and let $T$ be fixed in $\{1,\ldots,n\}$. Then for all real $x_0$,
\begin{align*}
\bigg| \mathbb{P}(W<x_0)&-\frac{1}{\sqrt{2\pi}}\int_{-\infty}^{x_0}e^{\frac{-x^2}{2}}dx \bigg| \\
&\leq 
\frac{v_1}{a}\sqrt{\sum_{j=1}^n \frac{p_2(j)^2}{v_j}\left(\frac{K_j(T)}{v_j}+1-\frac{2K_1(T)}{v_1}\right)^2} \\
&+\frac{\sqrt{v_1}}{\pi^{1/4}}\left[\sum_{j=0}^n\left(8-\frac{6}{a}\left(1-\frac{K_j(T)}{v_j}\right)\right)\frac{p_2(j)^2}{v_j}\right]^{1/4}.
\end{align*}
\end{theorem}  
Armed with this theorem, we can examine how varying the value of $T$ affects the error term.  However, rather than examining the error for
a fixed value of $T$, we define $T$ to be a random variable on $\{0,\ldots,n\}$ with $\mathbb{P} \, (T=t)=b_t$ (and $b_0 \not= 1$).  Note that
this modification does not affect the stationary distribution of the chain. 

There are two reasons for using a random variable instead of a fixed value of $T$.  The first reason is that using a random variable yields
a larger family of Markov chains.  The second reason is that when $q=2$, the Plancherel distribution is
a binomial distribution with parameters
$n$ and $p=1/2$.  In this case, Theorem \ref{LT}
implies the chain $L_t(i,j)$ chooses $t$ coordinates with probability $b_t$ and changes them all. That is,
$L_t(i,j)$
follows the rule of moving a (Hamming) distance $t$ away with probability $b_t$ and we recover the chain from Section \ref{binsec}.

We will show that the
random variable $T$ that minimizes the error term from Theorem \ref{stn86} is induced by the Markov chain used in \cite{ful05} which has $b_1=1$ (or
alternatively $b_0+b_1=1$ and $b_0 \not= 1$) and
all other $b_t=0$.
The computation of the error term with $T$ as a random variable closely follows the proof of Theorem
\ref{krawful05} from \cite{ful05}.   
To apply Theorem \ref{stn86}, we need to compute the terms in the bound on the error.  It will be helpful to define $(W_t,W_t')$
to be the exchangeabe pair defined in the usual way from the Markov chain $L_t(i,j)$.  Note that $W_t=W$ since using a different Markov chain
in the family does not
alter the stationary distribution.

%In order to compute the error term
%from Theorem \ref{stn86} for $(W,W')$ using work already done in \cite{ful05} on $(W_t,W_t')$, we use the following elementary proposition,

%\begin{proposition} \label{trn}
%Let $f$ be a function of two variables.  Then 
%\[\mathbb{E}[f(W,W')]=\sum_{t=0}^n b_t\mathbb{E}[f(W_t,W_t')]\]

%\end{proposition} 

The next two lemmas will generate all the terms needed to apply Theorem \ref{stn86}.

\begin{lemma}\cite{ful05} \label{fks}

\begin{align*}
\mathbb{E}[(W_t'-W_t)^2|i]& =v_1\sum_{r=0}^n\frac{K_r(i)K_r(t)}{v_r^2}p_2(r)+\left(1-\frac{2K_1(t)}{v_1}\right)W_t^2. \\
Var(\mathbb{E}[(W_t'-W_t)^2|i])&=v_1^2\sum_{r=1}^n\frac{p_2(r)^2}{v_r} \left( 1+\left( \frac{K_r(t)}{v_r}-\frac{2K_1(t)}{v_1} \right) \right)^2. \\
\mathbb{E}[(W_t'-W_t)^4]&=v_1^2\left[\sum_{r=0}^n\left(8\left(1-\frac{K_1(t)}{v_1}\right)-6\left(1-\frac{K_r(t)}{v_r}\right)\right)\frac{p_2(r)^2}{v_r}\right].
\end{align*}

\end{lemma}

\begin{lemma} \label{lamham}
\begin{align*}
\mathbb{E}[W'|W]&=\left(\sum_{t=0}^nb_t\frac{K_1(t)}{v_1} \right)W. \\
\mathbb{E}[(W'-W)^2|i] &=v_1\sum_{r=0}^n\frac{K_r(i)\left(\sum_{t=0}^nb_tK_r(t)\right)}{v_r^2}p_2(r) 
+\left(1-\frac{2\left(\sum_{t=0}^nb_tK_1(t)\right)}{v_1}\right)W^2. \\
Var(\mathbb{E}[(W'-W)^2|i])&=v_1^2\sum_{r=1}^n \frac{p_2(r)^2}{v_r} \left( 1+\sum_{t=0}^nb_t\left( \frac{K_r(t)}{v_r}-\frac{2K_1(t)}{v_1} \right) \right)^2. \\
\mathbb{E}[(W'-W)^4] & =v_1^2\bigg[\sum_{r=0}^n\bigg(8\bigg(1-\frac{\sum_{t=0}^nb_tK_1(t)}{v_1}\bigg) 
-6\bigg(1-\frac{\sum_{t=0}^nb_tK_r(t)}{v_r}\bigg)\bigg)\frac{p_2(r)^2}{v_r}\bigg]. \\
%\mathbb{E}[(W'-W)^2]&=2\left(1-\frac{\sum_{t=0}^nb_tK_1(t)}{v_1}\right).
\end{align*}
\end{lemma}

\begin{proof}
We prove only the first equality; the proofs of the remaining are similar. 
By summing over $t$ equal to the value of $T$ chosen, we have

\[\mathbb{E}[W'|i]=\frac{1}{\sqrt{v_1}}\sum_{t=0}^nb_t\sum_{j=0}^nL_t(i,j)K_1(j).\]
The equality now follows from the proof of Lemma \ref{krawlam}.

\end{proof}

Applying Lemma \ref{lamham} to Theorem \ref{stn86} proves the following theorem.

\begin{theorem} \label{krawerr}
If $W$ and $\{b_t\}_{t=0}^n$ are defined as above, $a=1-\left(\sum_{t=0}^nb_t\frac{K_1(t)}{v_1}\right)$, and all other variables are as in Theorem \ref{krawful05}, then for all real $x_0$,
\begin{align*}
&\left| \mathbb{P}(W<x_0)-\frac{1}{\sqrt{2\pi}}\int_{-\infty}^{x_0}e^{\frac{-x^2}{2}}dx \right| \\
&\leq 
\frac{v_1}{a}\sqrt{\sum_{j=1}^n \frac{p_2(j)^2}{v_j}\left(\frac{\sum_{t=0}^nb_tK_j(t)}{v_j}+1-\frac{2\left(\sum_{t=0}^nb_tK_1(t)\right)}{v_1}\right)^2} \\
&+\frac{\sqrt{v_1}}{\pi^{1/4}}\left[\sum_{j=0}^n\left(8-\frac{6}{a}\left(1-\frac{\sum_{t=0}^nb_tK_j(t)}{v_j}\right)\right)\frac{p_2(j)^2}{v_j}\right]^{1/4}.
\end{align*}
\end{theorem}  

We can now analyze the error term of Theorem \ref{krawerr} over $L_T(i,j)$, the family of Markov chains previously defined.

\begin{theorem}\label{krawopt}
The Markov chain $L_T(i,j)$ that minimizes the error term from Theorem \ref{krawerr} is the chain with $\mathbb{P}(T=1)=b_1=1$.
\end{theorem}  
\begin{proof}
The right hand side of the inequality of Theorem \ref{krawerr} can be rewritten as
\begin{align}
&v_1\sqrt{\sum_{j=1}^n \frac{p_2(j)^2}{v_j}\left(\frac{\left(\sum_{t=0}^nb_t\frac{K_j(t)}{v_j}\right)-1}{a}+2\right)^2} \label{krawt1}\\
&+\frac{\sqrt{v_1}}{\pi^{1/4}}\left[\sum_{j=0}^n\left(8+6\left(\frac{\left(\sum_{t=0}^nb_t\frac{K_j(t)}{v_j}\right)-1}{a}\right)\right)\frac{p_2(j)^2}{v_j}\right]^{1/4}\label{krawt2}.
\end{align}

We will show that
\begin{align}
&\frac{\left(\sum_{t=0}^nb_t\frac{K_j(t)}{v_j}\right)-1}{a} \label{krawt}
\end{align}
is minimum for each $j$ under some set $\mathcal{B}=\{b_t\}_{t=0}^n$,  which implies (\ref{krawt2}) is minimum under $\mathcal{B}$ (notice $v_j \geq 0$ for all $j$).
We will then show that the minimum value of (\ref{krawt}) is no less than $-2$ for each $j$, so that (\ref{krawt1}) is also minimum under $\mathcal{B}$.

Since $p_2(j)=0$ for $j\geq3$, we only consider $j=0,1,2$.  For $j=0$, $K_0(t)=v_0=1$, so that (\ref{krawt}) does not depend on $t$.
For $j=1$, the numerator of (\ref{krawt}) is equal to $-a$ so that (\ref{krawt}) is independent of $t$.  For $j=2$, a straightforward calculation (using $\sum_{t=0}^nb_t=1$) yields
\begin{align}
&\frac{\left(\sum_{t=0}^nb_t\frac{K_2(t)}{v_2}\right)-1}{a}+2
=\left(\frac{q}{q-1}\right)\frac{\sum_{t=2}^nb_t\frac{t(t-1)}{n(n-1)}}{\sum_{t=1}^nb_t\left(\frac{t}{n}\right)}\label{krawt3}.
\end{align}
Since the summand in the numerator is $0$ when $t=1$ and positive for all other $t$ values (where $b_t$ is positive),
the final term in (\ref{krawt3}) is minimum for $b_1+b_0=1$ 
and $b_t=0$ for $t>1$. 
\end{proof}

\begin{remarks}
\mbox{}
\begin{enumerate}
\item Similar to Section \ref{binsec}, Theorem \ref{LT} implies that the case $b_0+b_1=1$ corresponds to the Markov chain $L_T$ having step
size at most one and with associated value of $a$ from Theorem \ref{stn86} equal to $\frac{qb_1}{(q-1)n}$.  
Among chains with $b_0=0$ the chain yielding the best bound has the smallest maximum step size and the smallest value of $a$.
We refer to the remarks following Theorem \ref{binmthm} on this restriction.
\item It is interesting to note that the size of the bound on the error in this section and in Section \ref{binsec} both depend on the underlying chain
through the same term ($B/A$ from
Section \ref{binsec}).  This is obvious from the case $q=2$, since this case is the same in both sections, but it is not clear
why this should carry over to other values of $q$.
\end{enumerate}
\end{remarks}

%%%%%%%%%%%%%%%%%%%%%%%%%%%%%%%%%%%%%%%%%%%%%%%NEW SECTION%%%%%%%%%%%%%%%%%%%%%%%%%%%%%%%%%%%%%%%%%%%%%%%%%%%%%

%%%%%%%%%%%%%%%%%%%%%%%%%%%%%%%%%%%%%%%%%%%%%%%%%%%%%%%%%%%%%%%%%%%%%%%%%%%%%%%%%%%%%%%%%%%%%%%%%%%%%%%%%%%%%%%%%

%%%%%%%%%%%%%%%%%%%%%%%%%%%%%%%

					%%%%%%%%%%%%%%%%%%%%%%%%%%%%%%%%%%%%%%%%%%%%%%%%%%%%%%%%%%%

\section{Plancherel Distribution on a Group}\label{planchsec}

In this section we examine the Plancherel measure of the random walk generated by the conjugacy class of transpositions on $S_n$
(the symmetric group on $n$ symbols).  First we describe the setup for any group in order to state the theorems we will use in the
utmost generality. Let $G$ be a group and $C$ be a nontrivial conjugacy class of $G$ such that $C=C^{-1}$. 
Define the random walk on $G$ generated by $C$, a Markov chain with state space $G$,
as follows: given $g$ in $G$, the next step in the chain is $gh$ where $h$ is chosen
uniformly at random from $C$.
%given a state (an element of $G$), the next step in the chain is given by choosing an
%element of $C$ uniformly at random and multiplying on the left by the given state.
Now, denote the set of irreducible
representations of $G$ by $Irr(G)$.  From \cite{dish81}, for each character $\chi^{\lambda}$ of $\lambda$ in
$Irr(G)$, there is an eigenvalue of the random walk on $G$ generated by $C$ given
by $\frac{\chi^\lambda(C)}{dim(\lambda)}$ (termed a character ratio) occurring with multiplicity $dim(\lambda)^2$.
A well known fact from representation theory is (see for example \cite{sa01}) \[\sum_{\lambda \in Irr(G)}dim(\lambda)^2=|G|,\] so that we
define the Plancherel measure of $G$ to choose $\lambda$ in $Irr(G)$ 
with probability $\frac{dim(\lambda)^2}{|G|}$.  Then the unit variance random variable
$W(\lambda)=\frac{|C|^{1/2}\chi^{\lambda}(C)}{dim(\lambda)}$ is essentially an eigenvalue
of the random walk above chosen uniformly at random.

In \cite{ful05} the author defines a Markov chain reversible with respect to the Plancherel measure of a group $G$ with the probability of
transitioning from $\lambda$ to $\rho$ given by
\[L_\tau(\lambda,\rho)=\frac{dim(\rho)}{dim(\lambda)dim(\tau)}\frac{1}{|G|}\sum_{g \in G}\chi^\lambda(g)\chi^\tau(g)\overline{\chi^\rho(g)}.\]
Here $\tau$ is some representation of $G$.
Following the usual setup, define an exchangeable pair $(W,W')$ by $W=W(\lambda)$
where $\lambda$ is chosen from the Plancherel measure of $G$, and define $W'=W(\rho)$
where $\rho$ is given by one step in the above chain.  For the sake of continuity, we will postpone discussing
properties of this Markov chain family until later in this section.

Using facts from representation theory, \cite{ful05} proves the following lemma which allows for the application of Theorem \ref{stn86}.

\begin{lemma}\cite{ful05}\label{gplanchlem1}
$\mathbb{E}(W'|W)=\left(\frac{\chi^\tau(C)}{dim(\tau)}\right)W.$
\end{lemma}

\begin{theorem}\cite{ful05}\label{gplanch}
Let $C$ be a conjugacy class of a finite group $G$ such that $C=C^{-1}$ and fix a nontrivial irreducible representation $\tau$
of $G$ whose character is real valued.  Let $\lambda$ be a random irreducible representation chosen from the Plancherel measure
of $G$. Let $W=\frac{|C|^{1/2}\chi^\lambda(C)}{dim(\lambda)}$.  Then for all real $x_0$,
\begin{align*}
\bigg| \mathbb{P} (W \leq x_0)- & \frac{1}{\sqrt{2\pi}}\int_{-\infty}^{x_0}e^{-\frac{x^2}{2}}dx \bigg| \\
	& \leq |C| \sqrt{\sum_{K \not= id}\frac{p_2(K)^2}{|K|} \left( \frac{\frac{\chi^\tau(K)}{dim(\tau)}-1}{a_\tau}+2 \right) ^2} \\
	& + \left[ \frac{|C|^2}{\pi}\sum_{K}\frac{p_2(K)^2}{|K|}\left(8+6\left(\frac{\frac{\chi^\tau(K)}{dim(\tau)}-1}{a_\tau}\right)\right) \right]^{1/4}. \\
\end{align*}
Here the sums are over conjugacy classes $K$ of $G$, $a_\tau=1-\frac{\chi^\tau(C)}{dim(\tau)}$, and $p_2(K)$
is the probability that the random walk on $G$ generated by $C$ started at the identity is at the conjugacy class $K$ after two steps.
\end{theorem}

After proving this theorem, the author applies it with the choice of the irreducible representation $\tau$ corresponding to the partition $(n-1,1)$ 
(more on this notation later) to obtain a central limit theorem with an error term in the case where
$G=S_n$ and $C$ is the conjugacy class of i-cycles.  We will show that this choice of $\tau$ minimizes the error term from Theorem \ref{gplanch}
in the special case where $C$ is the conjugacy class of transpositions (2-cycles).

%Notice that similar to Section \ref{krawsec}, if we can show that for each conjugacy class $K$ of $S_n$, the quantity
%$\left(\frac{\chi^\tau(K)}{dim(\tau)}-1\right)\Big/a_\tau$ is minimized and greater than $-2$ when $\tau=(n-1,1)$, then the whole term
%will be minimized for that choice of irreducible representation. 

Before going further we will explain the notation above and state some facts about the irreducible representations of
$S_n$ found in \cite{sa01}.
For $S_n$, the irreducible representations are indexed by partitions of the integer $n$ where the partition $(n)$
corresponds to the
trivial representation.  Another way to index the irreducible representations (which can be more useful for combinatorial reasons) is
to associate to each partition a ``tableau" which is a left justified array of equally sized, aligned boxes (we will typically
abuse notation and refer to $\lambda$ as both the partition and the diagram).  For
a partition $\lambda=(\lambda_1,\lambda_2, \ldots , \lambda_k)$ of $n$, where $\lambda_1 \geq \ldots \geq \lambda_k>0$, the
associated tableau has $\lambda_1$ boxes in the first row, $\lambda_2$ boxes in the second row, and so on. Now, numbering the boxes one
through $n$
left to right, top to bottom (with this indexing, this is technically the largest Standard Young Tableau), we make the following definition.

\begin{definition}\label{condef}
Let $\lambda=(\lambda_1,\lambda_2, \ldots , \lambda_k)$ a partition of $n$.  Define the content of box $i$ (in the labeling above) of $\lambda$ to be
$c_\lambda(i)=column_\lambda(i)-row_\lambda(i)$.
\end{definition}

In order to clarify these definitions, Table \ref{TAB1} is the labeled tableau for $\lambda=(5,3,3,1)$, a partition of $n=12$.

\begin{table}[h]
\begin{center}
\begin{tabular}{|c|c|c|c|c|}
\hline
1&2&3&4&5\\
\hline
6&7&8\\
\cline{1-3}
9&10&11\\
\cline{1-3}
12\\
\cline{1-1}
\end{tabular}
\end{center}
\caption{Labeled tableau for $\lambda=(5,3,3,1)$.}
\label{TAB1}
\end{table}

Table \ref{TAB2} below is the same tableau, but the box labeled $i$ above now has the value of the content $c_\lambda(i)$.
\begin{table}[h]
\begin{center}
\begin{tabular}{|c|c|c|c|c|}
\hline
0&1&2&3&4\\
\hline
-1&0&1\\
\cline{1-3}
-2&-1&0\\
\cline{1-3}
-3\\
\cline{1-1}
\end{tabular}
\end{center}
\caption{$\lambda=(5,3,3,1)$ with box $i$ labeled $c_\lambda(i)$.}
\label{TAB2}
\end{table}

Because we are specializing to the conjugacy class $C$ of transpositions
and $p_2(K)$ is the probability of being in the conjugacy class $K$ after two steps
in the random walk on $S_n$ generated by $C$ starting at the identity element, $p_2(K)=0$ for many conjugacy classes.
The following lemma formulates
$\frac{\chi^\lambda(K)}{dim(\lambda)}$ in terms of the contents of $\lambda$
for conjugacy classes $K$ where $p_2(K)\not=0$
(the conjugacy classes that contribute to the error term in Theorem \ref{gplanch}).  The lemma is proved in the form shown here using
Murphy's elements in \cite{digr89}, but can also be found in terms of the $\lambda_i$ in \cite{ing50}.
\begin{lemma}\cite{digr89}\label{cont}
Let $\chi^\lambda(j)$ and $\chi^\lambda(2,2)$ be the character of the irreducible representation $\lambda$ at a the conjugacy class of a j-cycle 
and two 2-cycles, respectively. Then  

\begin{align*} 
\frac{\chi^\lambda(id)}{dim(\lambda)}&=1, \\
\frac{\chi^\lambda(2)}{dim(\lambda)}&=\frac{2}{n(n-1)}\sum_{i=1}^nc_\lambda(i), \\
\frac{\chi^\lambda(3)}{dim(\lambda)}&=\frac{3}{n(n-1)(n-2)}\left(\sum_{i=1}^nc_\lambda(i)^2-\binom{n}{2}\right), \\
\frac{\chi^\lambda(2,2)}{dim(\lambda)}&=\frac{1}{6\binom{n}{4}}\left(\left(\sum_{i=1}^nc_\lambda(i)\right)^2-3\sum_{i=1}^nc_\lambda(i)^2+n(n-1)\right). \\
\end{align*}
\end{lemma}

For the example from Tables \ref{TAB1} and \ref{TAB2} where $\lambda=(5,3,3,1)$, the sum of the contents is equal to four and $n$ is equal to twelve.
Thus Lemma \ref{cont} implies, for example, $\frac{\chi^\lambda(2)}{dim(\lambda)}=2/33$.

We pause here to discuss some relevant properties of the chain $L_\lambda$.  First, we reiterate the remarks at the end of Section \ref{NA}; the
definition of the family of Markov chains stems from the orthogonality relations of irreducible characters (see \cite{ful08} and the 
references therein).  
Second, because it is not known how to combinatorially describe $L_\lambda$ in general, it is not obvious how to relate the choice of $\lambda$
to step size. However, some representations do have nice descriptions and we can use these to gain some intuition.  From \cite{ful05a}, if 
$\nu$ is the defining representation on $S_n$, then given an irreducible representation
$\mu$, $L_\nu$ follows the rule of removing a corner box of $\mu$ uniformly at random (so that the resulting diagram of $n-1$ boxes remains
a tableau) and moving it to a uniformly chosen
concave open position of the altered
diagram (to obtain a new tableau of size $n$). For $\rho$ the trivial representation, $\nu$ the defining representation, and $\tau:=(n-1,1)$,
we have \cite{sa01} 
$\chi^\nu=\chi^\tau+\chi^\rho$, $dim(\rho)=1$, $dim(\nu)=n$, and $dim(\tau)=n-1$ so that
\[L_\tau=\frac{nL_\nu-L_\rho}{n-1}.\]
Now, using the orthogonality relations of characters and the fact that $\chi^\rho\equiv 1$ it is easy to see the chain $L_\rho$
holds with probability one. Thus, the previous remarks imply that from an irreducible representation $\mu$,
the chain $L_\tau$ moves at most one box of $\mu$.   
Also, again using orthogonality relations,
it follows that for any irreducible representations $\mu$ and $\lambda$, we have $L_{\lambda}(\rho,\mu)=\delta_{\lambda\mu}$.
That is, from the trivial representation the chain $L_{\lambda}$ moves to $\lambda$ with probability
one, which will move more than one box if $\lambda\not=\tau$.  Lemma \ref{cont} implies that $W(\mu)=\binom{n}{2}^{-1/2}\sum_{i=1}^nc_\mu(i)$,
so that moving more boxes of a partition to obtain $W'$ from $W$ corresponds to a larger step size. This is one sense in which
$\tau$ has the smallest step size among nontrivial irreducible representations.

As discussed previously, a more quantitative measure of step size of the chain $L_\lambda$ is
the value of $a_\lambda$ from Theorem \ref{gplanch}. From Lemma \ref{gplanchlem1}, we have
\begin{align}
\mathbb{E}(W'|W)&=\left(\frac{\chi^\lambda(2)}{dim(\lambda)}\right)W 
	=\left(\frac{2}{n(n-1)}\sum_{i=1}^nc_\lambda(i)\right)W. \notag
\end{align}
By the definition of $c_\lambda(i)$, it is easy to see that $a_\lambda$ is minimum over nontrivial irreducible representations 
for $\lambda=\tau$, and strictly increases as
boxes of a partition are moved from higher to lower rows (this observation motivates Lemma \ref{reduc} below).
This is another sense in which
$\tau$ has the smallest step size among nontrivial irreducible representations. 
%  that over the family of Markov chains $L_\lambda$ indexed by nontrivial irreducible representations,
% $L_\tau$ generates an exchangeable pair with coordinates nearest to each other.

Now, analogous to Section \ref{krawsec}, in order to show that the error term from Theorem \ref{gplanch}
is minimum over nontrivial irreducible representations at
$\tau=(n-1,1)$, we will show that for
each conjugacy class $K$ from Lemma \ref{cont}, the term $T_K(\lambda)
=\left(\frac{\chi^\lambda(K)}{dim(\lambda)}-1\right)\Big/a_\lambda$ 
is minimum and at least negative two for $\lambda=\tau$.  Moreover, we
will define an ordering on all irreducible representations such that
the largest nontrivial irreducible representation in the ordering is $(n-1,1)$ and then show that 
$T_K$ is a decreasing function with respect to the ordering.  This non-standard ordering is a coarsening of the usual dominance
ordering found in \cite{sa01}.

\begin{definition}\label{succ}
Let $\lambda=(\lambda_1, \ldots, \lambda_k)$ and $\mu=(\mu_1, \ldots, \mu_m)$ be two irreducible representations of $S_n$.
We say $\lambda$ succeeds $\mu$, denoted $\lambda \succ \mu$, if $\lambda_1 \geq \mu_1$, $\lambda_i=\mu_i$
for $i=2, \ldots, k-1$, and $\lambda_k \leq \mu_k$. For $i>m$ ($i>k$), take $\mu_i=0$ ($\lambda_i=0$).
\end{definition}

It is obvious that $(n-1,1)=\tau \succ \lambda$ for all nontrivial irreducible representations $\lambda$, so that to prove the claims above it is
enough to show that
$\lambda\succ\mu$ implies $T_K(\lambda) \leq T_K(\mu)$ for each $K$ from Lemma \ref{cont}.  The following lemma is
a simpler characterization of the succession relation.

\begin{lemma}\label{reduc}
For two irreducible representations $\lambda$ and $\mu$ of $S_n$, $\lambda\succ\mu$ if and only if $\mu$ can be transformed from
$\lambda$ only by moving blocks from the associated tableau of $\lambda$ from the top row to either the bottom row of $\lambda$
or to start a new row. 
\end{lemma}

\begin{proof}
If the defining relations of succession hold, and the two representations are not equal, then either $\lambda_k < \mu_k$ or
$\mu_{k+1} \not = 0$.  In the former case move a block from $\lambda_1$ to $\mu_k$, in the latter start a new row.
Inductively, continuing in
this way will eventually lead to $\mu$ since the
new tableau still succeeds $\mu$. 

Conversely, if $\mu$ can be made from $\lambda$ with the described method, then clearly the defining relations of succession must hold.
\end{proof}

Lemma \ref{reduc} states that in order to prove $\lambda\succ\mu$ implies $T_K(\lambda) \leq T_K(\mu)$ (thus that $\tau=(n-1,1)$ minimizes
the error term), it is enough to prove the
statement in the case $\mu=(\lambda_1-1,\lambda_2,\ldots,\lambda_k+1)$; where in order to cover both
actions described in Lemma \ref{reduc} we break convention and permit $\lambda_k=0$ (only if $\lambda_{k-1} \not = 0)$.  We make one final 
(non-standard) 
definition before
we begin the proof of the claims above.
\begin{definition}
Let $\lambda=(\lambda_1, \ldots, \lambda_k)$ and $\mu=(\lambda_1-1,\lambda_2,\ldots,\lambda_k+1)$.
Define a new tableau, the joint of $\lambda$ and $\mu$ by
\[\lambda|\mu=(\lambda_1-1,\lambda_2, \ldots, \lambda_k).\] 
\end{definition}

Now that we have suitable definitions, we can begin to prove the lemmas that will be used to obtain the main result.

\begin{lemma}
For any tableau $\lambda$, $T_{id}(\lambda)=0$ and $T_{(2)}(\lambda)=-1$.
\end{lemma}
\begin{proof}
The first assertion follows from the fact that for any representation, the identity element is mapped to the
identity matrix so that $\chi^\lambda(id)=dim(\lambda)$.  The second is obtained directly from the definition of $a_\lambda$.
\end{proof}

The next lemma states a simpler criterion for determining which of $T_{(3)}(\lambda)$ and $T_{(3)}(\mu)$ is larger.

\begin{lemma}\label{3simp}
With $\lambda$ and $\mu$ as above, $T_{(3)}(\mu) - T_{(3)}(\lambda)$ is non-negative if and only if $f(\lambda)$
(defined below) is also non-negative.
\begin{align}
f(\lambda)=\left(\lambda_1+\lambda_k-k\right)&\left(\sum_{i=1}^{n-1}c_{\lambda | \mu}(i) -\binom{n}{2}\right)+(\lambda_1-1)(\lambda_k+1-k) \notag\\
 &-\sum_{i=1}^{n-1}c_{\lambda | \mu}^2(i)+\binom{n}{2}+2\binom{n}{3}. \notag
\end{align}
\end{lemma}

\begin{proof}
\begin{align}
&T_{(3)}(\mu) - T_{(3)}(\lambda) \notag \\
&= \frac{\left(1-\frac{1}{2\binom{n}{3}}\left(\sum_ic_\lambda^2(i)-\binom{n}{2}\right)\right)a_{\mu}-
\left(1-\frac{1}{2\binom{n}{3}}\left(\sum_ic_{\mu}^2(i)-\binom{n}{2}\right)\right)a_{\lambda}}
{a_{\lambda} a_{\mu}}. \label{t3}
\end{align}
Notice that $\sum_ic_\lambda(i)$ is maximized for $\lambda=(n)$ in which case it is $\binom{n}{2}$ and strictly
less for all other tableau.
%(since $(n)$ succeeds all other tableau and each move lessens the sum of the contents). 
Thus, for all tableau
$\lambda$ under consideration, $a_\lambda>0$ so that the non-negativity of (\ref{t3}) is determined by the numerator.  The numerator of (\ref{t3})
is equal to
\begin{align}
&\frac{\sum_i(c_{\mu}^2(i)-c_\lambda^2(i))}{2\binom{n}{3}} + \frac{\sum_i(c_\lambda(i)-c_{\mu}(i))}{\binom{n}{2}} \notag \\
&+ \frac{\sum_ic_\lambda^2(i)\sum_jc_{\mu}(j)-\sum_ic_\lambda(i)\sum_jc_{\mu}^2(j)+\binom{n}{2}\sum_i(c_\lambda(i)-c_{\mu}(i))}
{2\binom{n}{3}\binom{n}{2}}. \label{3pos} 
\end{align}

Rewriting  the sums $\sum_ic_\lambda^t(i)=\sum_ic_{\lambda|\mu}^t(i)+(\lambda_1-1)^t$ and
$\sum_ic_{\mu}^t(i)=\sum_ic_{\lambda|\mu}^t(i)+(\lambda_k+1-k)^t$ for $t=1,2$ 
and then simplifying implies (\ref{3pos}) is equal to
\begin{align*}
& \frac{(\lambda_k+1-k)^2-(\lambda_1-1)^2}{2\binom{n}{3}} + \frac{(\lambda_1-1)-(\lambda_k+1-k)}{\binom{n}{2}} \\
&+ \frac{\left((\lambda_1-1)^2-(\lambda_k+1-k)^2\right)\sum_ic_{\lambda|\mu}(i)}{2\binom{n}{3}\binom{n}{2}}\\
&+\frac{(\lambda_1-1)(\lambda_k+1-k)\left((\lambda_1-1)-(\lambda_k+1-k)\right)}{2\binom{n}{3}\binom{n}{2}} \\
&-\frac{((\lambda_1-1)-(\lambda_k+1-k))\sum_ic_{\lambda|\mu}^2(i)}{2\binom{n}{3}\binom{n}{2}}\\
&+\frac{\binom{n}{2}((\lambda_1-1)-(\lambda_k+1-k))}{2\binom{n}{3}\binom{n}{2}}. \\
%& \Bigg[\left((\lambda_1-1)^2-(\lambda_k+1-k)^2\right)\sum_ic_{\lambda|\mu}(i) \\
%&+(\lambda_1-1)(\lambda_k+1-k)((\lambda_1-1)-(\lambda_k+1-k)) \\
%&-((\lambda_1-1)-(\lambda_k+1-k))\sum_ic_{\lambda|\mu}^2(i)\\
%&+\binom{n}{2}((\lambda_1-1)-(\lambda_k+1-k))\Bigg] \bigg/ \left(2\binom{n}{3}\binom{n}{2}\right) \\
\end{align*}
Multiplying by $\left(2\binom{n}{3}\binom{n}{2}\right)/((\lambda_1-1)-(\lambda_k+1-k))$ proves the lemma
(since $\lambda_1>\lambda_k$ and $k\geq2$, this multiplication does not affect the non-negativity of the term).
\end{proof}

Now we can prove the following lemma.

\begin{lemma}\label{3mon}
Let $\lambda$ and $\mu$ be two irreducible representations on the symmetric group.
If $\lambda\succ\mu$, then $T_{(3)}(\lambda) \leq T_{(3)}(\mu).$
\end{lemma}

\begin{proof}
Lemma \ref{3simp} implies that in order to show the monotonicity of $T_{(3)}$ with respect to the succession relation, we only need to show that $f(\lambda)\geq 0$ for all
tableau $\lambda=(\lambda_1, \ldots, \lambda_k)$ with $\lambda_1>\lambda_2$ and $\lambda_k<\lambda_{k-1}$ (where $\lambda_k$ may possibly be zero). In order to prove
the lemma, we use induction on the $n$, the number being partitioned.  In each of the three cases below, we associate to each tableau
$\lambda$
of size $n+1$ a tableau $\phi$ of size $n$ where $f(\phi) \geq 0$ by the induction hypothesis.  It can be easily verified that in the
case of $n=3$, the only nontrivial tableau satisfying $\lambda_1>\lambda_2$ and $\lambda_k<\lambda_{k-1}$ is $\lambda=(2,1,0)$, and
$f(\lambda)\geq 0$.

\vspace{2.5mm}
\noindent Case 1: $\lambda_k \not = 0$.

For this case, let $\phi=(\lambda_1,\ldots,\lambda_{k-1},\lambda_k-1)$ be a partition of $n$ and let $\psi$
be equal to the partition $(\lambda_1-1,\ldots,\lambda_{k-1},\lambda_k)$.
Then we have
\begin{align*}
f(\lambda)&=\left(\lambda_1+\lambda_k-k\right)\left(\sum_{i=1}^{n}c_{\lambda | \mu}(i) -\binom{n+1}{2}\right)\\
&+(\lambda_1-1)(\lambda_k+1-k) -\sum_{i=1}^{n}c_{\lambda | \mu}^2(i)+\binom{n+1}{2}+2\binom{n+1}{3},  \\ 
f(\phi)&=\left(\lambda_1+(\lambda_k-1)-k\right)\left(\sum_{i=1}^{n-1}c_{\phi | \psi}(i) -\binom{n}{2}\right) \\
&+(\lambda_1-1)((\lambda_k-1)+1-k) -\sum_{i=1}^{n-1}c_{\phi | \psi}^2(i)+\binom{n}{2}+2\binom{n}{3}. 
\end{align*}

Since $f(\phi)$ is non-negative by the induction hypothesis, we will show that $f(\phi)-f(\lambda)\leq 0$ which will
prove the lemma for this case.
Simplifying this expression using the identity $\binom{n+1}{k}=\binom{n}{k}+\binom{n}{k-1}$ and the fact that
\[\sum_{i=1}^{n-1}(c_{\phi|\psi}(i))^t=\sum_{i=1}^{n}(c_{\lambda|\mu}(i))^t-(\lambda_k-k)^t \,\,\,\,\,\,\,\,\, t=1,2,\]
we obtain
\[f(\phi)-f(\lambda)=(n-1)(\lambda_1-1)+(n-\lambda_1+1)(\lambda_k-k)-\sum_{i=1}^nc_{\lambda|\mu}(i)-\binom{n}{2}.\]

Now we bound the sum of the contents;
\begin{align*}
\sum_{i=1}^nc_{\lambda|\mu}(i) & \geq \sum_{i=1}^{\lambda_1-2}i+\sum_{i=1}^{\lambda_k}(i-k)+(n-(\lambda_1-1)-\lambda_k)(1-(k-1)) \\
									& = \frac{(\lambda_1-1)(\lambda_1-2)}{2}+\frac{\lambda_k(\lambda_k-3)}{2}+(2-k)(n-\lambda_1+1).
\end{align*}
The inequality follows because the first term on the right hand side
is the sum of the contents of row one, the second term is the sum of the contents of
row $k$, and the third term is the number of boxes not in rows one or $k$ times the minimum content of those boxes.

The inequality above shows that $f(\phi)-f(\lambda)\leq g(\lambda_1,\lambda_k)$ where the function $g$ is defined by:
 
\begin{align*}
g(\lambda_1,\lambda_k)&=(n-1)(\lambda_1-1)+(n-\lambda_1+1)(\lambda_k-2) \\
&-\frac{(\lambda_1-1)(\lambda_1-2)}{2} 
-\frac{\lambda_k(\lambda_k-3)}{2}-\binom{n}{2}.
\end{align*}

Let $\mathcal{D}=\{(\lambda_1,\lambda_k):\lambda_k \leq n+1-\lambda_1, 1\leq \lambda_k \leq \lambda_1-2\}$ and 
notice that the allowable
values of $(\lambda_1,\lambda_k)$ are a subset of $\mathcal{D}$.  A straightforward (but tedious) analysis shows
the maximum of $g$ over the domain $\mathcal{D}$ is non-positive, which proves the lemma for this case.

\vspace{2.5mm}

\noindent Case 2: $\lambda_k=0$, $\lambda_{k-1}\not=1$. 

In this case, let $\phi=(\lambda_1,\ldots,\lambda_{k-1}-1,0)$ and then let $\psi$ be equal to
the partition $(\lambda_1-1,\ldots,\lambda_{k-1}-1,1)$.  Then let 
%Then simplifying and defining $f(\phi)$ analogously to the first case, we have
%\begin{align*}
%f(\lambda)=\left(\lambda_1-k\right)&\left(\sum_{i=1}^{n}c_{\lambda | \mu}(i) -\binom{n+1}{2}\right)+(\lambda_1-1)(1-k) \\ 
% &-\sum_{i=1}^{n}c_{\lambda | \mu}^2(i)+\binom{n+1}{2}+2\binom{n+1}{3},  \\ 
\begin{align*}
f(\phi)=\left(\lambda_1-k\right)&\left(\sum_{i=1}^{n-1}c_{\phi | \psi}(i) -\binom{n}{2}\right)+(\lambda_1-1)(1-k) \\
 &-\sum_{i=1}^{n-1}c_{\phi | \psi}^2(i)+\binom{n}{2}+2\binom{n}{3},
\end{align*}
and notice $f(\phi)\geq0$ by the induction hypothesis. After simplifying using the fact that
\[\sum_{i=1}^{n-1}(c_{\phi|\psi}(i))^t=\sum_{i=1}^{n}(c_{\lambda|\mu}(i))^t-(\lambda_{k-1}-k+1)^t \,\,\,\,\,\,\,\,\, t=1,2,\]
we obtain
\begin{equation}
f(\phi)-f(\lambda)=n(\lambda_1-k)+(\lambda_{k-1}-\lambda_1+1)(\lambda_{k-1}-k+1)-n^2. \label{rfff}
\end{equation}

The right hand side of equation (\ref{rfff}) is increasing in $\lambda_1$, so that the maximum is attained for $\lambda_1=n$. In this case
it is easy to see (\ref{rfff}) is non-positive.

% If $\lambda_{k-1}-k+1 \geq 0$, then the term is negative since $n^2 \geq n(\lambda_1-k)$.  Assuming $\lambda_{k-1}-k+1 \leq 0$, and
% noting $\lambda_{k-1} \geq 2$, we have 
% \begin{align}
% f(\phi)-f(\lambda) \leq n(\lambda_1-k)+(3-\lambda_1)(3-k)-n^2. \label{fc2}
% \end{align}
% Treating this as a function of $\lambda_1$ and $k$ under the constraints $\lambda_1 \geq 3$, $k \geq 3$, and 
% $2(k-2)+\lambda_1 \leq n+1$, it is a simple to show (\ref{fc2}) is non-positive.  

% there are no extrema in the domain, and the boundaries $\lambda_1=3$ and
% $k=3$ are both clearly negative.  Setting $k=(n-\lambda_1+5)/2$ givs a downward facing parabola in $\lambda_1$
% with roots equal to $n+3$ and $3n-1$, so that the term is negative over the domain.

\vspace{2.5mm}

\noindent Case 3: $\lambda=(\lambda_1, \ldots, \lambda_{k-2},1,0)$.

In this case, let $\phi=(\lambda_1,\ldots,\lambda_{k-2},0)$ and let $\psi$ be equal to
the partition $(\lambda_1-1,\ldots,\lambda_{k-2},1)$.  Analogously to the previous two cases, define
%\begin{align*}
%f(\lambda)=\left(\lambda_1-k\right)&\left(\sum_{i=1}^{n}c_{\lambda | \mu}(i) -\binom{n+1}{2}\right)+(\lambda_1-1)(1-k) \\ 
% &-\sum_{i=1}^{n}c_{\lambda | \mu}^2(i)+\binom{n+1}{2}+2\binom{n+1}{3},  \\ 
\begin{align*}
f(\phi)=\left(\lambda_1-(k-1)\right)&\left(\sum_{i=1}^{n-1}c_{\phi | \psi}(i) -\binom{n}{2}\right)+(\lambda_1-1)(1-(k-1)) \\
 &-\sum_{i=1}^{n-1}c_{\phi | \psi}^2(i)+\binom{n}{2}+2\binom{n}{3},
\end{align*}
which implies
%Simplifying this expression using the identity $\binom{n+1}{k}=\binom{n}{k}+\binom{n}{k-1}$ and the fact that
%\[\sum_{i=1}^{n-1}(c_{\phi|\psi}(i))^t=\sum_{i=1}^{n}(c_{\lambda|\mu}(i))^t-(2-k)^t \,\,\,\,\,\,\,\,\, t=1,2,\]
%we obtain
\begin{align*}
f(\phi)-f(\lambda) =(n+k)(\lambda_1-n)+	\sum_{i=1}^{n-1}c_{\phi|\psi}(i)-\binom{n}{2} +3-2k-\lambda_1.
\end{align*}
Using that $\lambda_1\leq n$ and thus $\sum_{i=1}^{n-1}c_{\phi|\psi}(i)\leq \binom{n-1}{2}$,
it is easy to see the term is negative.

\end{proof}

Now we move on to the final conjugacy class needed to prove the result.  As before we have the following lemma which states a
simpler criterion for determining which of $T_{(2,2)}(\lambda)$ and $T_{(2,2)}(\mu)$ is larger.

\begin{lemma}\label{22simp}
With $\lambda$ and $\mu$ as above, $T_{(2,2)}(\mu) - T_{(2,2)}(\lambda)$ is non-negative if and only if $h(\lambda)$
(defined below) is also non-negative.
\begin{align*}
h(\lambda)&=2\left(\lambda_1+\lambda_k-k\right)\left(\binom{n}{2}-\sum_{i=1}^{n-1}c_{\lambda | \mu}(i) \right)-2(\lambda_1-1)(\lambda_k+1-k) \\
 &-2\binom{n}{2}\sum_{i=1}^{n-1}c_{\lambda | \mu}(i)+\left(\sum_{i=1}^{n-1}c_{\lambda | \mu}(i)\right)^2
+3\sum_{i=1}^{n-1}c_{\lambda | \mu}^2(i)+6\binom{n}{4}-2\binom{n}{2}. \\
\end{align*}
\end{lemma}

\begin{proof}
\begin{align}
T_{(2,2)}(\mu) - T_{(2,2)}(\lambda) 
= \frac{\left(1-\frac{\chi^\lambda(2,2)}{dim(\lambda)}\right)a_{\mu}-
\left(1-\frac{\chi^\mu(2,2)}{dim(\mu)}\right)a_{\lambda}}
{a_{\lambda} a_{\mu}}. \label{t222}
\end{align}

As in Lemma \ref{3simp}, the non-negativity of (\ref{t222}) is determined by the numerator.  Using the formulas from
Lemma \ref{cont} and rearranging, the numerator of (\ref{t222}) is equal to
\begin{align*}
&\frac{\left(\sum_ic_{\mu}(i)\right)^2-\left(\sum_ic_{\lambda}(i)\right)^2-3\left(\sum_i[c_{\mu}^2(i)
-c_\lambda^2(i)]\right)}{6\binom{n}{4}} 
+ \frac{\sum_i(c_\lambda(i)-c_{\mu}(i))}{\binom{n}{2}} \\
&+ \frac{\left[\left(\sum_ic_{\lambda}(i)\right)^2-3\sum_ic_\lambda^2(i)+2 \binom{n}{2}\right]\sum_jc_{\mu}(j)}{6\binom{n}{4}\binom{n}{2}} \\
&- \frac{\left[\left(\sum_ic_{\mu}(i)\right)^2-3\sum_ic_{\mu}^2(i)+2 \binom{n}{2}\right]\sum_jc_{\lambda}(j)}{6\binom{n}{4}\binom{n}{2}}.
\end{align*}

Analogous to Lemma \ref{3simp} for $3$-cycles, simplifying the term by
rewriting $\sum_ic_\lambda^t(i)=\sum_ic_{\lambda|\mu}^t(i)+(\lambda_1-1)^t$ and
$\sum_ic_{\mu}^t(i)=\sum_ic_{\lambda|\mu}^t(i)+(\lambda_k+1-k)^t$ for $t=1,2$
yields a term proportional to $h(\lambda)$.  In this case, the constant
of proportionality is $\left(6\binom{n}{4}\binom{n}{2}\right)/((\lambda_1-1)-(\lambda_k+1-k))$, which is
positive and so does not affect the non-negativity.
%\begin{align*}
%& \frac{\lambda_1+\lambda_k-k-\sum_ic_\lambda(i)}{3\binom{n}{4}}\\
%& \frac{(\lambda_k+1-k)^2-(\lambda_1-1)^2}{2\binom{n}{3}} + \frac{(\lambda_1-1)-(\lambda_k+1-k)}{\binom{n}{2}} \\
%&+ \frac{\left((\lambda_1-1)^2-(\lambda_k+1-k)^2\right)\sum_ic_{\lambda|\mu}(i)}{2\binom{n}{3}\binom{n}{2}}\\
%&+\frac{(\lambda_1-1)(\lambda_k+1-k)\left((\lambda_1-1)-(\lambda_k+1-k)\right)}{2\binom{n}{3}\binom{n}{2}} \\
%&-\frac{((\lambda_1-1)-(\lambda_k+1-k))\sum_ic_{\lambda|\mu}^2(i)}{2\binom{n}{3}\binom{n}{2}}\\
%&+\frac{\binom{n}{2}((\lambda_1-1)-(\lambda_k+1-k))}{2\binom{n}{3}\binom{n}{2}} \\
%& \Bigg[\left((\lambda_1-1)^2-(\lambda_k+1-k)^2\right)\sum_ic_{\lambda|\mu}(i) \\
%&+(\lambda_1-1)(\lambda_k+1-k)((\lambda_1-1)-(\lambda_k+1-k)) \\
%&-((\lambda_1-1)-(\lambda_k+1-k))\sum_ic_{\lambda|\mu}^2(i)\\
%&+\binom{n}{2}((\lambda_1-1)-(\lambda_k+1-k))\Bigg] \bigg/ \left(2\binom{n}{3}\binom{n}{2}\right) \\
%\end{align*}
\end{proof}

Now we can prove the following lemma.

\begin{lemma}\label{22mon}
Let $\lambda$ and $\mu$ be two irreducible representations on the symmetric group.
If $\lambda\succ\mu$, then $T_{(2,2)}(\lambda) \leq T_{(2,2)}(\mu)$.
\end{lemma}

\begin{proof}
We will exactly follow the strategy of the proof of Lemma \ref{3mon} with the function $h$ replacing $f$.

%Lemma \ref{22simp} implies that in order to show the proper monotonicity of $T_{(2,2)}$, we only need to show that $h(\lambda)\geq 0$ for all
%tableau $\lambda$ with $\lambda_1>\lambda_2$ and $\lambda_k<\lambda_{k-1}$ (where $\lambda_k$ may possibly be zero). In order to prove
%the lemma, we use induction on the $n$, the number being partitioned.  In each of the three cases below, we associate to each tableau
%$\lambda$
%of size $n+1$ a tableau $\phi$ of size $n$ where $h(\phi) \geq 0$ by the induction hypothesis.  It can be easily verified that in the
%case of $n=3$, the only nontrivial tableau satisfying $\lambda_1>\lambda_2$ and $\lambda_k<\lambda_{k-1}$ is $\lambda=(2,1,0)$, and
%$h(\lambda)\geq 0$.

\vspace{2.5mm}
\noindent Case 1: $\lambda_k \not = 0$.

For this case, let $\phi=(\lambda_1,\ldots,\lambda_{k-1},\lambda_k-1)$, a partition of $n$, and $\psi$
be equal to the partition $(\lambda_1-1,\ldots,\lambda_{k-1},\lambda_k)$.
Then following the proof of Lemma \ref{3mon},
%\begin{align*}
%&h(\lambda)=2\left(\lambda_1+\lambda_k-k\right)\left(\binom{n+1}{2}-\sum_{i=1}^{n}c_{\lambda | \mu}(i) \right)-2(\lambda_1-1)(\lambda_k+1-k) \\
% &-2\binom{n+1}{2}\sum_{i=1}^{n}c_{\lambda | \mu}(i)+\left(\sum_{i=1}^{n}c_{\lambda | \mu}(i)\right)^2
%+3\sum_{i=1}^{n}c_{\lambda | \mu}^2(i) 
%+6\binom{n+1}{4}-2\binom{n+1}{2}, \\
%&h(\phi)=2\left(\lambda_1+\lambda_k-1-k\right)\left(\binom{n}{2}-\sum_{i=1}^{n-1}c_{\phi | \psi}(i) \right)-2(\lambda_1-1)(\lambda_k-k) \\
% &-2\binom{n}{2}\sum_{i=1}^{n-1}c_{\phi|\psi}(i)+\left(\sum_{i=1}^{n-1}c_{\phi | \psi}(i)\right)^2
%+3\sum_{i=1}^{n-1}c_{\phi | \psi}^2(i)+6\binom{n}{4}-2\binom{n}{2}. \\
%\end{align*}

%Since $h(\phi)$ is non-negative by the induction hypothesis, we will show that $h(\phi)-h(\lambda)\leq 0$ which will
%prove the lemma for this case.
%Simplifying this expression using the identity $\binom{n+1}{k}=\binom{n}{k}+\binom{n}{k-1}$ and the fact that
%\[\sum_{i=1}^{n-1}(c_{\phi|\psi}(i))^t=\sum_{i=1}^{n}(c_{\lambda|\mu}(i))^t-(\lambda_k-k)^t \,\,\,\,\,\,\,\,\, t=1,2,\]
%we obtain
\begin{align*}
(h(\phi)&-h(\lambda))/2=(\lambda_1-1)(\lambda_k-k+1)+\lambda_k\binom{n}{2}-(k+1)\binom{n}{2} \\
&+(n-\lambda_k+k+1)\sum_{i=1}^nc_{\lambda|\mu}(i)-n(\lambda_1+\lambda_k-k-1)-3\binom{n}{3} \\
&\leq (\lambda_1-1)(\lambda_k-k+1)+\binom{n}{2}\left(\lambda_k-k-1\right)\\
&+(n-\lambda_k+k+1)\binom{n-1}{2}-n(\lambda_1+\lambda_k-k-1)-3\binom{n}{3} \\
&=(\lambda_1-2)(\lambda_k-k)-\lambda_1(n-1)\leq0.
\end{align*}
Here the first inequality follows because $n-\lambda_k+k+1\geq0$ and $\sum_ic_{\lambda|\mu}(i) \leq \binom{n-1}{2}$,
and the final inequality from $\lambda_k-k \leq n-1$ and $\lambda_1 \geq 2$.

\vspace{2.5mm}

\noindent Case 2: $\lambda_k=0$, $\lambda_{k-1}\not=1$. 

In this case, let $\phi=(\lambda_1,\ldots,\lambda_{k-1}-1,0)$ and then let $\psi$ be equal to
the partition $(\lambda_1-1,\ldots,\lambda_{k-1}-1,1)$.  Then we have
%\begin{align*}
%&h(\lambda)=2\left(\lambda_1-k\right)\left(\binom{n+1}{2}-\sum_{i=1}^{n}c_{\lambda | \mu}(i) \right)-2(\lambda_1-1)(1-k) \\
% &-2\binom{n+1}{2}\sum_{i=1}^{n}c_{\lambda | \mu}(i)+\left(\sum_{i=1}^{n}c_{\lambda | \mu}(i)\right)^2
%+3\sum_{i=1}^{n}c_{\lambda | \mu}^2(i) 
%+6\binom{n+1}{4}-2\binom{n+1}{2}, \\
%&h(\phi)=2\left(\lambda_1-k\right)\left(\binom{n}{2}-\sum_{i=1}^{n-1}c_{\phi | \psi}(i) \right)-2(\lambda_1-1)(1-k) \\
% &-2\binom{n}{2}\sum_{i=1}^{n-1}c_{\phi|\psi}(i)+\left(\sum_{i=1}^{n-1}c_{\phi | \psi}(i)\right)^2
%+3\sum_{i=1}^{n-1}c_{\phi | \psi}^2(i)+6\binom{n}{4}-2\binom{n}{2}. \\
%\end{align*}

%Simplifying this expression using the identity $\binom{n+1}{k}=\binom{n}{k}+\binom{n}{k-1}$ and the fact that
%\[\sum_{i=1}^{n-1}(c_{\phi|\psi}(i))^t=\sum_{i=1}^{n}(c_{\lambda|\mu}(i))^t-(\lambda_{k-1}-k+1)^t \,\,\,\,\,\,\,\,\, t=1,2,\]
%we obtain
\begin{align*}
(h(\phi)&-h(\lambda))/2=(\lambda_1-k)(\lambda_{k-1}-k+1-n)+\binom{n}{2}\left(\lambda_{k-1}-k+1\right) \\
&+(n-\lambda_{k-1}+k-1)\sum_{i=1}^nc_{\lambda|\mu}(i)-(\lambda_{k-1}-k+1)^2-3\binom{n}{3}+n \\
&\leq (\lambda_1-k)(\lambda_{k-1}-k+1-n)+\binom{n}{2}\left(\lambda_{k-1}-k+1\right) \\
&+(n-\lambda_{k-1}+k-1)\binom{n-1}{2}-(\lambda_{k-1}-k+1)^2-3\binom{n}{3}+n \\
&=(\lambda_1-\lambda_{k-1}-2)(\lambda_{k-1}-k+1-n).
\end{align*}
The inequality follows because $n-\lambda_{k-1}+k-1\geq0$ and $\sum_ic_{\lambda|\mu}(i) \leq \binom{n-1}{2}$.
If $\lambda_1-\lambda_{k-1}-2\geq0$, then this case is shown as $\lambda_{k-1}-k+1-n \leq0$.  If not, then
$\lambda_1=\lambda_{k-1}+1$ which implies $\lambda_j=\lambda_{k-1}$ for $j=2 \ldots k-1$ (since $\lambda_1\geq \lambda_2+1$).
In this case, $(k-1)\lambda_{k-1}=n$ and 
\[\sum_ic_{\lambda|\mu}(i)=\sum_{i=1}^{k-1}\sum_{j=1}^{\lambda_{k-1}}j-i%=\frac{\lambda_{k-1}(k-1)}{2}(\lambda_{k-1}+1-k)
=\frac{n}{2}(\lambda_{k-1}+1-k).\]
Then we have
\begin{align*}
(h(\phi)&-h(\lambda))/2=(\lambda_{k-1}+1-k)(\lambda_{k-1}-k+1-n)+\binom{n}{2}\left(\lambda_{k-1}-k+1\right) \\
&+\frac{n}{2}(\lambda_{k-1}+1-k)(n-\lambda_{k-1}+k-1)-(\lambda_{k-1}-k+1)^2-3\binom{n}{3}+n.
\end{align*}
This is a downward facing parabola in $\lambda_{k-1}$ with roots equal to $n+k-1$ and $n+k-4$, neither of which is a possible value
of $\lambda_{k-1}$ since $k$ must be at least three.

\vspace{2.5mm}

\noindent Case 3: $\lambda=(\lambda_1, \ldots, \lambda_{k-2},1,0)$.

In this case, let $\phi=(\lambda_1,\ldots,\lambda_{k-2},0)$ and then let $\psi$ be equal to
the partition $(\lambda_1-1,\ldots,\lambda_{k-2},1)$.  Then we have
%\begin{align*}
%&h(\lambda)=2\left(\lambda_1-k\right)\left(\binom{n+1}{2}-\sum_{i=1}^{n}c_{\lambda | \mu}(i) \right)-2(\lambda_1-1)(1-k) \\
% &-2\binom{n+1}{2}\sum_{i=1}^{n}c_{\lambda | \mu}(i)+\left(\sum_{i=1}^{n}c_{\lambda | \mu}(i)\right)^2
%+3\sum_{i=1}^{n}c_{\lambda | \mu}^2(i) 
%+6\binom{n+1}{4}-2\binom{n+1}{2}, \\
%&h(\phi)=2\left(\lambda_1-(k-1)\right)\left(\binom{n}{2}-\sum_{i=1}^{n-1}c_{\phi | \psi}(i) \right)-2(\lambda_1-1)(1-(k-1)) \\
% &-2\binom{n}{2}\sum_{i=1}^{n-1}c_{\phi|\psi}(i)+\left(\sum_{i=1}^{n-1}c_{\phi | \psi}(i)\right)^2
%+3\sum_{i=1}^{n-1}c_{\phi | \psi}^2(i)+6\binom{n}{4}-2\binom{n}{2}. \\
%\end{align*}

%Simplifying this expression using the identity $\binom{n+1}{k}=\binom{n}{k}+\binom{n}{k-1}$ and the fact that
%\[\sum_{i=1}^{n-1}(c_{\phi|\psi}(i))^t=\sum_{i=1}^{n}(c_{\lambda|\mu}(i))^t-(2-k)^t \,\,\,\,\,\,\,\,\, t=1,2,\]
%we obtain
\begin{align}
(h&(\phi)-h(\lambda))/2 \notag\\
&=(n+k-3)\left(\sum_{i=1}^nc_{\lambda|\mu}(i)+k-\lambda_1-\binom{n-1}{2}\right) \notag\\
&+n(4-k) -2\lambda_1+k-(2-k)^2\label{hh2}\\
&\leq (n-1)(4-\lambda_1)+k(2-\lambda_1)\label{hh1}
\end{align}
To see the inequality, notice that $n+k-3\geq 0$ and $\sum_ic_{\lambda|\mu}(i)\leq \binom{n-1}{2}$.

The term (\ref{hh1}) is clearly negative for $\lambda_1\geq4$, and the smaller cases can be individually handled
starting from (\ref{hh2})
(for example $\lambda_1=2$ implies $\sum_ic_{\lambda|\mu}(i)=-\binom{n}{2}$ and $k=n+1$).
\end{proof}

The main result of this section is stated in the following.

\begin{theorem}
If $\lambda \succ \mu$ for nontrivial irreducible representations $\lambda$ and $\mu$, then the error term from
Theorem \ref{gplanch} associated to $\lambda$ is no larger than the error
term associated to $\mu$. In particular, the error term is minimum for $\tau=(n-1,1)$.
\end{theorem}
\begin{proof}
By the previous lemmas and remarks, the theorem will be proved if $T_{(3)}(\tau)$ and $T_{(2,2)}(\tau)$ are no less than negative two 
(for $K$ equal to the identity or transposition conjugacy class, $T_K$ is constant).
Using the formulas from Lemma \ref{cont}, or the fact \cite{sa01} that  for
any permutation $g$,
$\chi^\tau(g)$ is equal to the number of fixed points of $g$ minus one and $dim(\tau)=n-1$, we have
\begin{align*}
&T_{(3)}(\tau) =-3/2,  \\
&T_{(2,2)}(\tau) = -2.
\end{align*}  
\end{proof}

%%%%%%%%%%%%%%%%%%%%%%%%%%%%%%%%%%%%%%%%%%%%%%%%%%%%%%%NEW SECTION%%%%%%%%%%%%%%%%%%%%%%%%%%%%%%%%%%%%%%%%%%%%%%%%
%%%%%%%%%%%%%%%%%%%%%%%%%%%%%%%%%%%%%%%%%%%%%%%%%%%%%%%%%%%%%%%%%%%%%%%%%%%%%

%%%%%%%%%%%%%%%%%%%%%%%%%%%%%%%%%%%%%%%%%%%%%%%%%%%%%%%%%%%%%%%%%%%%%%%%%%%%%%%%%%%%%%%%%%%%%%%%%%%%%%%%%%%%%%%%%%%%%%%%%%%

\section{Eigenvalue Characterization}\label{lamsec}

Heuristically, increasing the step size of a Markov chain has the effect of making the chain become ``random'' faster. 
This suggests that for an ergodic chain, the step size may be related to the rate of convergence to stationarity.
From basic facts about reversible Markov chains on a finite state space \cite{bre99}, all the chains previously defined are ergodic if and only if 
they are irreducible and aperiodic.  In this case, the rate of convergence to stationarity is determined by the eigenvalues, where the eigenvalues
with the largest moduli make the largest asymptotic contributions. Because of this relationship, we will express the error terms from the past sections 
in terms of the eigenvalues of the Markov chains used to induce the exchangeable pairs.
%For example, in the case of the binomial distribution with the chain defined in Section \ref{binsec},
% this occurs if there exist integers
% $j$ and $k$ such that $b_{2j+1}>0$ and $b_{2k}>0$.  In the case where the chains are ergodic, the convergence to stationary occurs at geometric rate 
% with base equal to $|\lambda_{(1)}|$, where $\lambda_{(1)}$ is the eigenvalue with the second largest modulus. Because of this relationship, we will 
%calculate the bound on the error in terms of the eigenvalues of the chains.
% In order to obtain these expressions, the first thing we need to do is calculate what the eigenvalues are.
 
We first examine the chains from Section \ref{krawsec} (and hence also Section \ref{binsec} as previously mentioned).
For a fixed value of $t$, the eigenvalues for $L_t(i,j)$ have been computed in \cite{ful05},
but we include the proof since it is illustrative.  Recall the definitions and notation from Section \ref{krawsec}.

\begin{lemma}\cite{ful05}
For fixed $t \in \{0,1, \ldots ,n\}$, the eigenvalues of $L_t(i,j)$ are $\frac{K_s(t)}{v_s}$ for $s=0\ldots n$.  
\end{lemma}

\begin{proof}
By definitions and Lemma \ref{kraw2},
\begin{align*}
% \mathbb{E}[K_s|i] 
\sum_{j=0}^nL_t(i,j)K_s(j) &= \frac{1}{q^n}\sum_{r=0}^n\frac{K_r(i)K_r(t)}{v_r^2}\sum_{j=0}^nv_jK_s(j)K_r(j) \\
				  &= \left(\frac{K_s(t)}{v_s}\right)K_s(i).	
\end{align*}
In other words, if $M$ is the transition matrix of $L_t(i,j)$ and $\textbf{v}$ is the vector with coordinate $j$ equal to $K_s(j)$,
then $M\textbf{v}=\frac{K_s(t)}{v_s}\textbf{v}$.
\end{proof}

A proof of the next lemma follows along the lines of the proof of Lemma \ref{lamham}. 

\begin{lemma} \label{Llam}
The eigenvalues of $L_T(i,j)$ as defined in Section \ref{krawsec} are for $s=0,\ldots,n$,
\[\lambda_s=\sum_{t=0}^nb_t\left(\frac{K_s(t)}{v_s}\right).\]
\end{lemma}

Recall from the proof of Theorem \ref{krawopt} the bound on the error depended only on the following term for $s=0,1,2$, 
\begin{equation}
\frac{\left(\sum_{t=0}^nb_t\frac{K_s(t)}{v_s}\right)-1}{1-\left(\sum_{t=0}^nb_t\frac{K_1(t)}{v_1}\right)}.\label{krr3}
\end{equation}
By Lemma \ref{Llam}, (\ref{krr3}) can be rewritten as $(\lambda_s-1)/(1-\lambda_1)$ for $s=0,1,2$.  For $s=0$ and $s=1$, (\ref{krr3})
does not depend on values $b_t$, so that we have the following theorem.

\begin{theorem}
The size of the error term given by Stein's method via the family of chains from Section \ref{krawsec} is a monotone increasing
function of
\[\frac{\lambda_2-1}{1-\lambda_1},\] 
where $\lambda_1$ and $\lambda_2$ are defined as in Lemma \ref{Llam}.
\end{theorem}

A natural problem that arises is to ascertain 
under what setting the eigenvalues that determine the error term will be the eigenvalues
with largest moduli (not including $\lambda_0=1$),
as these are the eigenvalues with the largest contribution to the rate of convergence.  It follows immediately from
definitions that in the case  $b_0+b_1=1$ where the error term is minimum, $\lambda_s=1-\frac{b_1qs}{n(q-1)}$,
so that the ordering of the eigenvalues
corresponds to the subscript notation. In this case the eigenvalues that affect
the error term will be the largest in moduli if and only if $b_1\leq\frac{2n(q-1)}{q(n+2)}$.

For the chains $L_{\tau}$ from Section \ref{planchsec}, the eigenvalues have already been computed in \cite{ful05} using orthogonality
relations.  Recall the definitions and notation from Section \ref{planchsec}.

\begin{lemma} \label{Ltau}\cite{ful05}
The eigenvalues of $L_{\tau}$ as defined in section \ref{planchsec} are for conjugacy classes $C$ of the group $G$,
\[\lambda_C=\frac{\chi^{\tau}(C)}{dim{\tau}}.\]
\end{lemma}

Also, from Theorem \ref{gplanch} and the remarks following it, the bound on the error term depends only on the following
term for the conjugacy classes $K=(id), (2), (3)$, and $(2,2)$, 
\begin{equation}
\frac{\frac{\chi^{\tau}(K)}{dim{\tau}}-1}{1-\frac{\chi^{\tau}(2)}{dim{\tau}}}.\label{pll3}
\end{equation}

By Lemma \ref{Ltau} (\ref{pll3}) can be rewritten as $(\lambda_K-1)/(1-\lambda_{(2)})$ for $K=(id), (2), (3)$, and $(2,2)$.  
For $K=(id)$ and $K=(2)$, (\ref{pll3})
does not depend on the irreducible representation $\tau$ used to generate $L_{\tau}$, so that we have the following theorem.

\begin{theorem}
The size of the error term given by Stein's method via the family of chains from Section \ref{planchsec} is a monotone increasing
function of
\[\frac{\lambda_K-1}{1-\lambda_{(2)}},\] 
where $\lambda_K$ is defined as in Lemma \ref{Ltau} and $K=(3)$ or $K=(2)(2)$.
\end{theorem}

Once again it is natural to ask for a given irreducible representation $\tau$, whether the three eigenvalues that affect the error term
are those with the largest moduli (not including $\lambda_{(id)}=1$).
For $\tau=(n-1,1)$, the representation where the error term is minimum, it is well known \cite{sa01} that 
\[\frac{\chi^{\tau}(K)}{dim{\tau}} = \frac{F_K-1}{n-1},\] 
where $F_K$ is defined to be the number of fixed points of $K$.  In this case, the eigenvalues that affect the error term are the eigenvalues with
the second, third, and fourth largest moduli.

%%%%%%%%%%%%%%%%%%%%%%%%%%%%%%%%%%%%%%%%%%%%%%%%%%%%%%%%%%%%%%%%%%%%%%%%%%%%%%%%%%%%%%%%%%%%%%%%%%%%%%%%%%%%%%%%%%%%%%%%%%%%%

%%%%%%%%%%%%%%%%%%%%%%%%%%%%%%%%%%%%%%%%%%%%%%

%%%%%%%%%%%%%%%%%%%%%%%%%%%%%%%%%%%%%%%%%%%%%%%%%%%%%%%%%%%%%%%%%%%%%%%%%%%%%%%%%%%%%%%%%%%%%%%%%%%%%%

%%%%%%%%%%%%%%%%%%%%%%%%%%%%%%%%%%%%%%%%%%%%%%%%NEW                SECTION!!!!!!!!!!!!!!!!%%%%%%%%%%%%%%%%%%%%%%%%%%%%%%%

\section{Poisson Approximation}\label{poisintro}

For the final two sections we change focus from approximation by the normal distribution to approximation by the Poisson distribution.
A metric routinely used for integer supported random variables $X$ and $Y$ is the total variation distance defined by
\[\norm{X-Y}_{TV}=\frac{1}{2}\left|\sum_{k\in\mathbb{Z}}\left(\mathbb{P}(X=k)-\mathbb{P}(Y=k)\right)\right|.\]
We will examine how the step size of the Markov chain that induces an exchangeable pair affects the error term
in the following theorem.

\begin{theorem}
\cite{cdm05}
\label{cdm05}
Let $W=\sum_iX_i$, a sum of random variables, such that $\mathbb{E}(W)=\lambda$. 
Let $(W,W')$ an exchangeable pair, $c$ any constant, and
$Poi_\lambda$ denote the Poisson distribution with mean $\lambda$.  Then for $C_\lambda$ a constant only
depending on $\lambda$,
\begin{align}
&\norm{W-Poi_\lambda}_{TV}   \notag \\
&   \leq C_\lambda\left[\mathbb{E} \big| \lambda -c\,\mathbb{P}(W'=W+1|\{X_i\}) \big| 
+\mathbb{E} \big| W -c\,\mathbb{P}(W'=W-1|\{X_i\}) \big| \right]. \label{cdme}
\end{align}
\end{theorem}

\begin{remarks}
\mbox{}
\begin{enumerate}
\item It has been shown \cite{roe08} that a result similar to Theorem \ref{cdm05} but with additional error terms
still holds assuming only that $W$ and $W'$ are equally distributed. 

\item Ideally, $c$ should be chosen so that
we have the approximate equalities 
\begin{align}
\mathbb{P}(W'=W+1|\{X_i\})&\approx\frac{\lambda}{c}, \label{apin1}\\
\mathbb{P}(W'=W-1|\{X_i\})&\approx\frac{W}{c}. \label{apin2}
\end{align}
It is shown in \cite{cdm05} that intuitively the existence of
such a constant is likely, a heuristic that is reinforced in the examples presented there.
In fact, if $(W'-W) \in \{-1,0,1\}$ and $\mathbb{E}(W'-W|W)=-a(W-\mathbb{E}(W))$,
it is easy to see that for the choice of $c=1/a$ we have the same error in 
the approximate equalities (\ref{apin1}) and (\ref{apin2}).  This is in general a useful guide for the choice of the constant $c$ (for more on
this line of thought see \cite{roe07}).
\end{enumerate}
\end{remarks}

One of the main technical details in analyzing (\ref{cdme}) for different exchangeable pairs, is the choice of the constant $c$. 
It would be preferable to have a systematic method of choosing this constant based on the exchangeable pair so that the results here
are not contrived.  Ideally, we would choose the constant $c$ to minimize the error terms from Theorem \ref{cdm05}, or more feasibly, their
Cauchy-Schwarz bound (choose $c$ to make the expectation of the terms in the absolute value signs zero).  However, in the examples presented here,
we will choose the constant $c$ to yield the best possible bound under the constraint that the
terms in the absolute value signs are positive.  Admittedly, part of the reason for this restriction is technical convenience, but
in practice choosing the constant in this way is typical (see the examples of \cite{cdm05}).  In the next section we will compare the error terms
using both of these strategies in a small example.

Both of the examples presented here are sums of i.i.d. random variables (Bernoulli and geometric).  Even the simplest introduction
of dependence (e.g. the hypergeometric distribution) yield results that make the type of analysis in this paper difficult. Because we are in the
setting of independence, we can use the same family of Markov chains for both examples. Given a vector $(X_1, \ldots, X_n)$ of
non-negative integer valued i.i.d. random variables,
the next step in the chain follows the rule of choosing
$k$ coordinates uniformly at random and replacing them with $k$ new i.i.d. random variables (with the same distribution as the original).
It is not hard to see that this chain is reversible with respect to vectors of i.i.d. random variables and hence generates an exchangeable
pair. Extending this exchangeable pair to the sum of the components of the vector allows for the application of Theorem \ref{cdm05}.

% \begin{lemma}
% Let $\textbf{X}=\{X_i\}_{i=1}^n$ be a random vector of
% non-negative integer valued i.i.d. random variables.
% Also, let $\textbf{x}=\{x_i\}_{i=1}^n$ and $\textbf{y}=\{y_i\}_{i=1}^n$ be two vectors of non-negative integers
% and $P_{\textbf{x}\textbf{y}}$ be the
% the probability that $\textbf{x}$ goes to $\textbf{y}$ in the above chain.  Then
% $\mathbb{P}(\textbf{X}=\textbf{x})P_{\textbf{x}\textbf{y}}=\mathbb{P}(\textbf{X}=\textbf{y})P_{\textbf{y}\textbf{x}}$. 
% \end{lemma}

% \begin{proof}
% If we let $P_c$ denote the probability that the coordinates where $\textbf{x}$ and $\textbf{y}$ differ are selected to be replaced and
% the coordinates selected where $\textbf{x}$ and $\textbf{y}$ agree stay agreed, then we have
% \begin{align*}
% \mathbb{P}(\textbf{X}=\textbf{x})P_{\textbf{x}\textbf{y}}&=P_c\prod_{i=1}^n\mathbb{P}(X_i=x_i)\prod_{x_i\not=y_i}\mathbb{P}(X_i=y_i)\\
% 									&=P_c\prod_{x_i\not=y_i}\mathbb{P}(X_i=x_i)\prod_{i=1}^n \mathbb{P}(X_i=y_i)\\
% 									&=\mathbb{P}(\textbf{X}=\textbf{y})P_{\textbf{y}\textbf{x}}.
% \end{align*}									
% \end{proof}

Finally, notice that under this chain it is not clear how modifying the number of coordinates chosen to be selected (i.e. varying $k$) will
affect the error term.

%%%%%%%%%%%%%%%%%%%%%%%%%%%%%%%%NEW SECTION!!!!!!!!!!!!!               %%%%%%%%%%%%%%%%%%%%%%%%%%%%%%%%%%%%%%

%%%%%%%%%%%%%%%%%%%%%%%%%%%%%%					$$$$$$$$$$$$$$$$$LKJLKLKJ$LKJLKJ$KLJ$KLJ

%%%%%%%%%%%%%%%%%%%%%%%%%%%%%%%%%%%%%%%%%%%%%%%%%%%%%%%%%%%%%%%%%%%%%%%%%%%%%%%%%%%%%%%%%%%%%%%%%%%%%%%%%%%%%%%%%%%%%%%%

%%%%%%%%%%%%%%%%%%%%%%%%%%%%%%%%5  NEW SECTION%%%%%%%%%%%%%%%%%%%%%%%%%%%%%%%%%%%%%

%%%%%%%%%%%%%%%%%%%5                      %%%%%%%%%%%%%%%%%%%%%%%        %%%%%%%%%%%%%%%%%%%%%%%%%%%%%

                    %%%%%%%%%%%%%%%%%%%%%%%%%%%%%%%
       %%%%%%%%%%%%%%%%%%%%%                                %%%%%%%%%%%%%%%%%%%%%%%%%%%%%%%%5

\section{Binomial Distribution}\label{binpoisec}

It is well known that the binomial distribution with parameters $n$ and $p$ converges to a Poisson distribution with
mean $\lambda$ as $n$ tends to infinity if $np$ tends to $\lambda$.  For simplicity, in this section we consider the case where
$p=1/n$, so that $\lambda=1$.  In this case we will show that among the exchangeable pairs associated
with the family of Markov chains described in Section \ref{poisintro} the term from Theorem \ref{cdm05} is minimized when $k=1$.
First we will prove some lemmas that will be used to
compute the error term from the theorem.

\begin{lemma}\label{pdef}
Let $\mathbb{P}_k$ denote probability under the chain that substitutes $k$ coordinates as described in Section \ref{poisintro}. Then
\begin{align*}
\mathbb{P}_k(W'=W+1|\{X_i\})&=\sum_{i=0}^{k-1}\binom{k}{i+1}\frac{(n-1)^{k-i-1}}{n^k}\frac{\binom{W}{i} \binom{n-W}{k-i}}{\binom{n}{k}},\\
\mathbb{P}_k(W'=W-1|\{X_i\})&=\sum_{i=1}^{k}\binom{k}{i-1}\frac{(n-1)^{k-i+1}}{n^k}\frac{\binom{W}{i} \binom{n-W}{k-i}}{\binom{n}{k}}.
\end{align*}
\end{lemma}

\begin{proof}
Let the random variable $Y$ be the number of ones in the $k$ coordinates chosen. Then $\mathbb{P}_k(W'=W+1|Y=i)$ is the probability of
$i+1$ ones in the binomial distribution with parameters $k$ and $p=1/n$, which implies
\begin{align*}
\mathbb{P}_k(W'=W+1|\{X_i\})&=\sum_{i=0}^{k-1}\mathbb{P}_k(W'=W+1|Y=i)\mathbb{P}_k(Y=i)\\
					&=\sum_{i=0}^{k-1}\binom{k}{i+1}\frac{(n-1)^{k-i-1}}{n^k}\frac{\binom{W}{i} \binom{n-W}{k-i}}{\binom{n}{k}}.\\
\end{align*}
Conditioning and summing over $Y$ also yields the expression for $\mathbb{P}_k(W'=W-1|\{X_i\})$ in the lemma.
\end{proof}

For the remainder of the section define
\begin{align}
c_k=\left(\frac{n}{k}\right)\left(\frac{n}{n-1}\right)^{k-1}. \label{cbin}
\end{align}

The next two lemmas prove a useful property of the constant $c_k$.

\begin{lemma}\label{W+}
$1-c_k\mathbb{P}_k(W'=W+1|\{X_i\})\geq0$.
\end{lemma}

\begin{proof}
Conditioning and summing over the random variable $Y$ equal to the number of ones chosen in the $k$ coordinates,
\begin{align*}
\mathbb{P}_k(W'=W&+1|\{X_i\})=\sum_{i=0}^{k-1}\binom{k}{i+1}\left(\frac{n-1}{n}\right)^{k-i-1}\left(\frac{1}{n}\right)^{i+1}
\mathbb{P}_k(Y=i) \\
					&= \sum_{i=0}^{k-1}\left(\frac{k}{n(i+1)}\right)\left(\frac{n-1}{n}\right)^{k-1}\binom{k-1}{i}\left(\frac{1}{n-1}\right)^{i}
\mathbb{P}_k(Y=i) \\
					&\leq \sum_{i=0}^{k-1}\left(\frac{k}{n}\right)\left(\frac{n-1}{n}\right)^{k-1}\mathbb{P}_k(Y=i) \\
					&\leq \frac{k}{n}\left(\frac{n-1}{n}\right)^{k-1}.
\end{align*}
Here the first inequality follows from the fact that $k\leq n$. 
% so the sum is a quantity smaller than one multiplied by a probability.
% Now multiplying by $c_k$ and simplifying implies
% \begin{align*}
% c_k\mathbb{P}_k(W'=W+1|\{X_i\}) \leq 1.
% \end{align*}
\end{proof}

\begin{lemma}\label{W-}
$W-c_k\mathbb{P}_k(W'=W-1|\{X_i\})\geq0$.
\end{lemma}

\begin{proof}
For $W=0$, the lemma is trivially true.  For $W \not = 0 $,
we condition and sum over the random variable $Y$ equal to the number of ones chosen in the $k$ coordinates to obtain
\begin{align*}
\mathbb{P}_k&(W'=W-1|\{X_i\})=\sum_{i=1}^{k}\mathbb{P}_k(W'=W-1|Y=i)
\left[\frac{\binom{W}{i}\binom{n-W}{k-i}}{\binom{n}{k}} \right] \\
					&=\sum_{i=1}^k\binom{k}{i-1}\left(\frac{n-1}{n}\right)^{k-i+1}\left(\frac{1}{n}\right)^{i-1}
\left(\frac{Wk}{n}\right)\left[\frac{\binom{W-1}{i-1}\binom{n-W}{k-i}}{i\binom{n-1}{k-1}} \right] \\
					&\leq \left(\frac{Wk}{n}\right)\left(\frac{n-1}{n}\right)^{k-1}\sum_{i=1}^k\binom{k}{i-1}\frac{1}{n(n-1)^{i-2}}
\left[\frac{\binom{W-1}{i-1}\binom{(n-1)-(W-1)}{(k-1)-(i-1)}}{\binom{n-1}{k-1}} \right] \\
					&\leq \left(\frac{Wk}{n}\right)\left(\frac{n-1}{n}\right)^{k-1}.
\end{align*}
To see the final inequality, notice that for each summand, the second part of the product is a probability of the hypergeometric
distribution and the first part of the product is at most one.

% Multiplying by $c_k$ and simplifying implies
% \begin{align*}
% c_k\mathbb{P}_k(W'=W-1|\{X_i\}) \leq W.
% \end{align*}
\end{proof}

The previous two lemmas show that for $c=c_k$ defined by (\ref{cbin}), the terms within the absolute values in (\ref{cdme}) are positive.
Under this constraint, note that the error terms are decreasing in $c_k$ and that $1-c_k\mathbb{P}_k(W'=W+1|W=0)=0$.
These observations imply that among
constants satisfying Lemmas \ref{W+} and \ref{W-}, the error from Theorem \ref{cdm05} is minimized for each $k$ when $c_k$ is defined as (\ref{cbin}).
As discussed in the previous section, this is a natural way to choose the constant in the approximation that allows for the analysis done here.

We pause here to show in a simple example the difference in the error terms from Theorem \ref{cdm05} 
when choosing the constant $c$ according to the two approaches outlined in Section \ref{poisintro}.  First, we will
determine the error terms using the strategy we take here in the case $k=1$.  For the following computations, recall that $p=1/n$.
We have
\begin{align}
1-c_1\mathbb{P}_1(W'=W+1|\{X_i\})&=1-(n-W)p=pW, \notag \\
W-c_1\mathbb{P}_1(W'=W-1|\{X_i\})&=W-W(1-p)=pW, \notag
\end{align}
which implies the error from Theorem \ref{cdm05} is equal to $2p$ (from \cite{cdm05}, $C_\lambda=1$ for $\lambda\leq 1$).
Choosing instead $c_1'=n/(1-p)$, so that $c_1'\mathbb{P}_1(W'=W+1)=1$ (this was discussed as the alternative system of choosing $c$), we obtain
\begin{align}
1-c_1'\mathbb{P}_1(W'=W+1|\{X_i\})&=1-\frac{(n-W)p}{1-p}=\frac{p(W-1)}{1-p}, \notag \\
W-c_1'\mathbb{P}_1(W'=W-1|\{X_i\})&=W-W=0, \notag
\end{align}
which implies the error from Theorem \ref{cdm05} is equal to
\[\frac{p\mathbb{E}|W-1|}{1-p}\leq\frac{p}{\sqrt{1-p}}.\]
In the limit, the two error terms differ in quality only by a constant and $c_1$ is asymptotically equal to $c_1'$. 
Although the Cauchy-Schwarz approach using $c_1'$ asymptotically yields a better constant, computing the appropriate moment information for general $k$
using this scheme is much more difficult than the strategy we choose.
Also, this small example suggests that the Cauchy-Schwarz approach will yield superior asymptotic rates only in the constant, so
that it is not worth the extra effort of computing the more complicated (and higher) moment information
needed in order to undertake the type of analysis presented in this paper.
Finally, we note that using the chain here (with $k=1$), 
it is possible to use intermediate terms in the proof of Theorem \ref{cdm05} with the constant $c_1$ to obtain the superior upper bound of $p$ 
\cite{cdm05}, however this approach does not carry over to the chains with larger step size.

Moving forward, in order to apply the theorem, we need to take the expected value of the terms in Lemma \ref{pdef}. 
% In order to find this expectation,
% we will need some moment information about $W$.
The next lemma
has a nice expression for the expectation we need.

\begin{lemma}\label{pmom}
\[\mathbb{E}\left[\binom{W}{i}\binom{n-W}{k-i}\right]i!(k-i)!=n\cdots(n-k+1)\frac{(n-1)^{k-i}}{n^k}.\]
\end{lemma}

\begin{proof}
\begin{align}
\mathbb{E}(s^Wr^{n-W})&=r^n\mathbb{E}\left[\left(\frac{s}{r}\right)^W\right] \notag\\
					  %&=r^n\left(\frac{s}{rn}+\left(\frac{n-1}{n}\right)\right)^n 
					  &=\left(\frac{s}{n}+r\left(\frac{n-1}{n}\right)\right)^n.  \label{binmo}
\end{align}
Taking $i$ derivatives with respect to $s$ and $k-i$ derivatives with respect to $r$ of (\ref{binmo}) and evaluating at
$r=s=1$ implies the lemma.
\end{proof}

The final lemma in this section establishes bounds on the error from Theorem \ref{cdm05}.

\begin{lemma}\label{pord}
Both of $\mathbb{E}[c_k\mathbb{P}_k(W'=W+1|\{X_i\})]$ and $\mathbb{E}[c_k\mathbb{P}_k(W'=W-1|\{X_i\})]$ are bounded above by
\[\left(\frac{n-1}{n}\right).\]
%\begin{align*}
%\mathbb{E}[\mathbb{P}(W'=W+1|W)]& \geq \frac{(n-1)(n-3)\cdots(n-(2k-3)(n-(2k-1))}{n^{k}} \\
%\mathbb{E}[\mathbb{P}(W'=W-1|W)] &\geq \frac{(n-1)(n-3)\cdots(n-(2k-3)(n-(2k-1))}{n^{k}}
%\end{align*}
\end{lemma}
\begin{proof}
From Lemmas \ref{pdef} and \ref{pmom} we have
\begin{align*}
\mathbb{E}&[c_k\mathbb{P}_k(W'=W+1|\{X_i\})]=c_k\sum_{i=0}^{k-1}\binom{k}{i+1}\binom{k}{i}\left(\frac{(n-1)^{2k-2i-1}}{n^{2k}}\right) \\
								&= c_k\left(\frac{n-1}{n}\right)^k\left(\frac{k}{n}\right) 
\sum_{i=0}^{k-1}\binom{k-1}{i}^2\left(\frac{(n-1)^{k-i-1}}{n^{k-1}(n-1)^i}\right)\left(\frac{k}{(i+1)(k-i)}\right) \\
								&=\left(\frac{n-1}{n}\right)
\sum_{i=0}^{k-1}\binom{k-1}{i}\left(\frac{(n-1)^{k-i-1}}{n^{k-1}}\right)\left(\binom{k-1}{i}\frac{1}{(n-1)^i}\right)\left(\frac{k}{(i+1)(k-i)}\right) \\
								&\leq\left(\frac{n-1}{n}\right).
\end{align*}
The inequality follows from the fact that each summand is the product of two terms no larger than one and a probability.

For the remaining term, exchangeability implies
$\mathbb{E}[c_k\mathbb{P}_k(W'=W-1|\{X_i\})]=\mathbb{E}[c_k\mathbb{P}_k(W'=W+1|\{X_i\})]$, which proves the lemma.
\end{proof}

\begin{theorem}\label{pinc}
For the values of $c_k$ defined previously, the error term from Theorem \ref{cdm05} is minimized for $k=1$ and is equal to $2p$.
\end{theorem}

\begin{proof}
The final bound
was computed previously in this section and the fact that it is minimum follows directly from Lemma \ref{pord} and the easily verified fact
\[\mathbb{E}[c_1\mathbb{P}_1(W'=W+1|\{X_i\})]=\mathbb{E}[c_1\mathbb{P}_1(W'=W-1|\{X_i\})]=\frac{n-1}{n}.\]
\end{proof}

%%%%%%%%%%%%%%%%%%%%%%%%%%%%%%%%%%%%%%%%%%%%%%%%%%%%%%%%%%%%%%%%%%%%%%%%%%%%%%%%%%%%%%%%%%%%%%%%%%%%%%%%%%%%%%%%%%%%%%%%%%%%%

%%%%%%%%%%%%%%%%%%%%%%%%%%%%%%%%%%%%%%%%%%%%%%%%%%%%%%%%%%%%%%%%%%%%%%%%%%%%%%%%%%%%%%%%%%%%%%%%%%%%%%%%%%%%%%%%%%%%%%%%%%%%%
%%%%%%%%%%%%%%%%%%%%%%%%%%%%%%%%%%%%%%%%%%%%%%%%%%%%%%%%%%%%%%%%%%%%%%%%%%%%%%%%%%%%%%%%%%%%%%%%%%%%%%%%%%%%%%%%%%%%%%%%%%%%%
%%%%%%%%%%%%%%%%%%%%%%%%%%%%%%%%%%%%%%%%%%%%%%%%%%%%%%%%%%%%%%%%%%%%%%%%%%%%%%%%%%%%%%%%%%%%%%%%%%%%%%%%%%%%%%%%%%%%%%%%%%%%%
%%%%%%%%%%%%%%%%%%%%%%%%%%%%%%%%%%%%%%%%%%%%%%%%%%%%%%%%%%%%%%%%%%%%%%%%%%%%%%%%%%%%%%%%%%%%%%%%%%%%%%%%%%%%%%%%%%%%%%%%%%%%%
%%%%%%%%%%%%%%%%%%%%%%%%%%%%%%%%%%%%%%%%%%%%%%%%%%%%%%%%%%%%%%%%%%%%%%%%%%%%%%%%%%%%%%%%%%%%%%%%%%%%%%%%%%%%%%%%%%%%%%%%%%%%%

%%%%%%%%%%%%%%%%%%%%%%%%%%%%%%%%%%%%%%%%%%%%%%%%%%%%%%%%%%%%%%%%%%%%%%%%%%%%%%%%%%%%%%%%%%%%%%%%%%%%%%%%%%%%%%%%%%%%%%%%%%%%%

%%%%%%%%%%%%%%%%%%%%%%%%%%%%%%%%%%%%%%%%%%%%%%%%%%%%%%%%%%%%%%%%%%%%%%%%%%%%%%%%%%%%%%%%%%%%%%%%%%%%%%%%%%%%%%%%%%%%%%%%%%%%%

%%%%%%%%%%%%%%%%%%%%%%%%%%%%%%%%%%%%%%%%%%%%%%%%%%%%%%%%%%%%%%%%%%%%%%%%%%%%%%%%%%%%%%%%%%%%%%%%%%%%%%%%%%%%%%%%%%%%%%%%%%%%%

%%%%%%%%%%%%%%%%%%%%%%%%%%%%%%%%%%%%%%%%%%%%%%%%%%%%%%%%%%%%%%%%%%%%%%%%%%%%%%%%%%%%%%%%%%%%%%%%%%%%%%%%%%%%%%%%%%%%%%%%%%%%%

%%%%%%%%%%%%%%%%%%%%%%%%%%%%%%%%%%%%%%%%%%%%%%%%%%%%%%%%%%%%%%%%%%%%%%%%%%%%%%%%%%%%%%%%%%%%%%%%%%%%%%%%%%%%%%%%%%%%%%%%%%%%%

\section{Negative Binomial Distribution}\label{NBpois}

The final example presented in this paper is the approximation of the negative binomial distribution by the Poisson. A random
variable $X$ has the geometric distribution with parameter $p$ if $\mathbb{P}(X=i)=(1-p)^ip$ for all non-negative integers $i$.
Classically, the random variable $X$ is viewed as the number of failures before the first success in a sequence of independent Bernoulli trials
each with probability of success equal to $p$.  The random variable $W$ is negative binomial with
parameters $r$ and $p$ if $W=\sum_{i=1}^rX_i$, where the $X_i$ are independent geometric random variables with parameter $p$.  By viewing $W$
as the number of failures before
$r$ successes have occurred in a sequence of Bernoulli trials, it is easy to see that for all non-negative integers $i$,
$\mathbb{P}(W=i)=\binom{r+i-1}{i}(1-p)^ip^{r}$. We will use Theorem \ref{cdm05} to approximate $W$ by $Poi(\lambda)$ where
$\lambda$ is the mean of $W$ equal to $r(1-p)/p$ (the mean of a geometric random variable is $(1-p)/p$).  

For fixed $\lambda$, $p=r/(\lambda+r)$ so that  
\begin{equation}\label{NBapp}
\mathbb{P}(W=i)=\frac{\lambda^i}{i!}\frac{(r+i-1)!}{(r-1)!(\lambda+r)^i}\left(1+\frac{\lambda}{r}\right)^{-r}.
\end{equation}
As $r$ goes to infinity, the distribution converges to a Poisson distribution with mean $\lambda$. However,
for fixed $\lambda$, $p$ approaches one as $r$ goes to infinity, so that when $p$ is small the negative binomial
will not be approximately Poisson.  Because of this fact,
in this example we will not obtain a result as straightforward as Theorem \ref{pinc}. For some values of $p$, the optimal
error term does not occur with the smallest step size.  We will prove all supporting lemmas for general $p$, but the final theorem will have a
natural restriction on the value of $p$.  For this case we will show that among the exchangeable pairs associated
with the family of Markov chains described in Section \ref{poisintro} the term from Theorem \ref{cdm05} is minimized when $k=1$.
First we will prove some lemmas that will be used to
compute the error term from the theorem.

\begin{lemma}\label{NBdef}
Let $\mathbb{P}_k$ denote probability under the chain that substitutes $k$ coordinates as described in Section \ref{poisintro}. Then
\begin{align*}
\mathbb{P}_k(W'=W+1|\{X_i\})& =\\
\binom{r}{k}^{-1}\mathop{\sum_{\{i_1,\ldots,i_k \}}}_{\subseteq \{1,\ldots,r\}}&\binom{k+\sum_jX_{i_j}}{k-1}(1-p)^{(\sum_jX_{i_j}+1)}p^k\\
\mathbb{P}_k(W'=W-1|\{X_i\})& =\\
\binom{r}{k}^{-1}\mathop{\sum_{\{i_1,\ldots,i_k \}}}_{\subseteq \{1,\ldots,r\}}
&\binom{k+\sum_j X_{i_j}-2}{k-1}(1-p)^{(\sum_j X_{i_j}-1)}p^k.\\
\end{align*}
\end{lemma}
\begin{proof}
This follows immediately from conditioning and summing over the subset of $(X_1,\ldots,X_r)$ chosen.
\end{proof}

For the remainder of the section define 
\begin{align}
c_k=\lambda a_k\left(k\binom{k+a_k-2}{k-1}(1-p)^{a_k}p^{k-1}\right)^{-1} \label{cneg}
\end{align}
where $a_k$ is the maximum of one and integer part of $\frac{(k-2)(1-p)}{p}$.  The next lemma states a useful property of $a_k$;
the proof can be found in \cite{jkk05}, but it is elementary so we will include it.

\begin{lemma}\label{ak}
Let $r$ be a positive integer, $x$ any non-negative integer, and $a$ be the
integer part of $\frac{(r-1)(1-p)}{p}$. If $f(y)=\binom{r+y-1}{r-1}(1-p)^{y}p^{r}$, then $f(x)\geq f(x-1)$ for $x\leq a$ and strictly decreasing otherwise.  In particular, $a$ is the mode of a negative binomial random variable with parameters $r$ and $p$.
%\[\binom{r+a-1}{r-1}(1-p)^{a}p^r\geq\binom{r+x-1}{r-1}(1-p)^{x}p^r.\]
\end{lemma}
\begin{proof}
The ratio of $f$ evaluated at consecutive integers is given by
\[\frac{f(x+1)}{f(x)}=\left(\frac{x+r}{x+1}\right)(1-p).\]
Comparing this ratio to one implies the lemma.
\end{proof}

%Taking $r=k-1$ in the previous lemma shows the maximum of $\binom{k+x-2}{k-2}(1-p)^{x}p^{k-1}$ occurs at $x=a_k$.% 

The next two lemmas prove a useful property of the constant $c_k$ as defined by (\ref{cneg}).
\begin{lemma}\label{NB+}
$\lambda-c_k\mathbb{P}_k(W'=W+1|\{X_i\})\geq0.$
\end{lemma}
\begin{proof}
\begin{align*}
c_k\mathbb{P}_k(W'&=W+1|\{X_i\}) = \\
&=\lambda\left(\frac{a_k}{k}\right)\binom{r}{k}^{-1} \mathop{\sum_{\{i_1,\ldots,i_k \}}}_{\subseteq \{1,\ldots,r\}}
\frac{\binom{k+\sum_jX_{i_j}}{k-1}(1-p)^{(\sum_jX_{i_j}+1)}p^k}
{\binom{k+a_k-2}{k-1}(1-p)^{a_k}p^{k-1}}.\\
\end{align*}
Define $b$ to be the maximum of one and the integer part of $\frac{(k-1)(1-p)}{p}$ and note that $b=1$ implies $a_k=1$. Lemma \ref{ak} implies
\begin{align*}
c_k\mathbb{P}_k(W'=W&+1|\{X_i\}) \\
&\leq \lambda\left(\frac{a_k}{k}\right)\binom{r}{k}^{-1} \mathop{\sum_{\{i_1,\ldots,i_k \}}}_{\subseteq \{1,\ldots,r\}}
\frac{\binom{k+b-1}{k-1}(1-p)^{b}p^{k}}
{\binom{k+a_k-2}{k-1}(1-p)^{a_k}p^{k-1}} \\
&=\lambda\left(\frac{(k+b-1)p}{k}\right)
\frac{\binom{k+b-2}{k-2}(1-p)^{b}}
{\binom{k+a_k-2}{k-2}(1-p)^{a_k}}\leq \lambda.\\
\end{align*}
The final inequality follows from the definition of $b$ and by applying Lemma
\ref{ak} with $r=k-1$. 
\end{proof}

\begin{lemma}\label{NB-}
$W-c_k\mathbb{P}_k(W'=W-1|\{X_i\})\geq0$.
\end{lemma}

\begin{proof}
The cases $W=0$ or $k=1$ are simple to verify, so assume otherwise.
By Lemma \ref{NBdef} and the definition of $c_k$,
\begin{align*}
c_k\mathbb{P}_k&(W'=W-1|\{X_i\}) \\
&=\mathop{\sum_{\{i_1,\ldots,i_k \}}}_{\subseteq \{1,\ldots,r\}}\frac{\left(\sum_jX_{i_j}\right)}{\dbinom{r-1}{k-1}}
\frac{\dbinom{k+\sum_jX_{i_j}-2}{k-2}(1-p)^{(\sum_jX_{i_j})}p^{k-1}}
{\dbinom{k+a_k-2}{k-2}(1-p)^{a_k}p^{k-1}}\\
&\leq \dbinom{r-1}{k-1}^{-1}\mathop{\sum_{\{i_1,\ldots,i_k \}}}_{\subseteq \{1,\ldots,r\}}\left(\sum_jX_{i_j}\right)=\sum_{i=1}^rX_i=W.\\
\end{align*}
An application of Lemma \ref{ak} with $r=k-1$ implies the inequality.
\end{proof}

The previous two lemmas show that for $c=c_k$ defined by (\ref{cneg}), the terms within the absolute values in (\ref{cdme}) are positive. 
Also note that 
\[c_k\mathbb{P}_k(W'=W-1|X_1=a_k,X_i=0; i>1)=a_k=W \not= 0,\]
so that among
constants satisfying Lemmas \ref{NB+} and \ref{NB-}, the error from Theorem
\ref{cdm05} is minimized for each $k$ when $c_k$ is defined as (\ref{cneg}).

To apply the theorem, we need to take the expected value of the term in Lemma \ref{NBdef}. 
%In order to find this expectation,
%we will need some moment information about $\sum_jX_{i_j}$.  
%We will use the fact that these sums are in fact negative binomial with parameters $p$ and $k$ in 
The next lemma has a nice expression for the expectation we need.

\begin{lemma}
If $Y$ is a random variable distributed as negative binomial with parameters $p$ and $k$, then
\[\mathbb{E}\left[\binom{k+Y}{k-1}(1-p)^{Y}\right]=\frac{\sum_{l=0}^{k-1}\binom{k}{l+1}\binom{k+l-1}{l}\left(\frac{(1-p)^2}{p(2-p)}\right)^l}{(2-p)^k}.\]
\end{lemma}

\begin{proof}
By the definition of expected value,
\begin{align}
\mathbb{E}\left[\binom{k+Y}{k-1}(1-p)^{Y}\right]&=\sum_{i\geq 0} (1-p)^{i}\binom{k+i}{k-1}(1-p)^i\binom{k+i-1}{k-1}p^k  \notag \\
		&=\frac{\sum_{i} \binom{k+i}{k-1}\binom{k+i-1}{k-1}((1-p)^{2})^{i}(p(2-p))^k}{(2-p)^k}. \label{expneg}
\end{align}
If $Z$ is a random variable distributed as negative binomial with parameters and $q=p(2-p)$ and $k$, then (\ref{expneg}) can be written as $\mathbb{E}\left[\binom{k+Z}{k-1}\right]/(2-p)^k$.  Using the fact that 
$Z$ is the sum of independent geometric random variables we have
\begin{align*}
\mathbb{E}(s^{k+Z})&=s^k\left(\frac{q}{1-(1-q)s}\right)^k =q^ks^k\left(\frac{1}{1-(1-q)s}\right)^k.
\end{align*}
Taking $k-1$ derivatives with respect to $s$ and dividing by $(k-1)!$ implies
\begin{align}
\mathbb{E}\left[\binom{k+Z}{k-1}s^{Z+1}\right]&=  \notag \\
\frac{q^k}{(k-1)!}\sum_{l=0}^{k-1}\binom{k-1}{l}&\frac{k!}{(l+1)!}s^{l+1}(1-q)^l
\frac{(k+l-1)!}{(k-1)!}\left(\frac{1}{1-(1-q)s}\right)^{k+l}. \label{expnegf}
\end{align}
Finally, substituting $s=1$ into (\ref{expnegf}) implies the lemma.
%\[\mathbb{E}\left[\binom{k+Z}{k-1}\right]=\sum_{l=0}^{k-1}\binom{k}{l+1}\binom{k+l-1}{l}\left(\frac{1-q}{q}\right)^l.\]
\end{proof}

The final results of this section will be stated in two cases.  The first case will pertain to ``small" values of $k$ where
$(k-1)(1-p)/p < 1$, and the ``large" case to all other values of $k$.  For fixed $k$ and $\lambda$, the small case is in some sense the
typical case as $p$ should be near one in order for $W$ to be approximately Poisson.  In this case, there is no need for further restrictions
on the value of $p$ in order to prove results analogous to the previous section. However, in the large case additional assumptions will be 
made.  We will first state and prove results for the small case, then discuss the additional assumptions and prove results for the large case.

\begin{lemma}\label{NBords}
For $\frac{(k-1)(1-p)}{p} < 1$,
both of $\mathbb{E}[c_k\mathbb{P}_k(W'=W+1|\{X_i\})]$ and $\mathbb{E}[c_k\mathbb{P}_k(W'=W-1|\{X_i\})]$ are bounded above by
\[\frac{r(1-p)}{2-p}.\]
\end{lemma}

\begin{proof}
For $k=1,2$, the lemma is easy to verify, so assume $k\geq3$.  
Let $Y$ be a random variable distributed as negative binomial with parameters $p$ and $k$. Then we have
\begin{align}
c_k\mathbb{E}[\mathbb{P}_k&(W'=W+1|\{X_i\})] \notag\\
&= c_k(1-p)p^k\mathbb{E}\left[\binom{k+Y}{k-1}(1-p)^{Y}\right]\notag\\
&=c_k(1-p)p^k\frac{\sum_{l=0}^{k-1}\binom{k}{l+1}\binom{k+l-1}{l}\left(\frac{(1-p)^2}{p(2-p)}\right)^l}{(2-p)^k} \label{nbb}\\
&=\frac{r(1-p)}{2-p}\left(\frac{1}{k(2-p)^{k-1}}\right)\sum_{l=0}^{k-1}\binom{k}{l+1}\left(\frac{1-p}{p}\right)^l\binom{k+l-1}{k-1}
\left(\frac{1-p}{2-p}\right)^{l}. \notag
\end{align}

Noting first that $(k-1)(1-p)<1$, an application of Lemma \ref{ak} implies $\binom{k+l-1}{k-1}
\left(\frac{1-p}{2-p}\right)^{l}$ is at most one, yielding the following inequality.

\begin{align*}
c_k\mathbb{E}[\mathbb{P}_k&(W'=W+1|\{X_i\})] \\
&\leq \frac{r(1-p)}{2-p}\left(\frac{1}{k(2-p)^{k-1}}\right)\sum_{l=0}^{k-1}\binom{k}{l+1}\left(\frac{1-p}{p}\right)^l \\
&=\left(\frac{r(1-p)}{2-p}\right)\left(\frac{1-p^k}{k(1-p)(p\,(2-p))^{k-1}}\right).
\end{align*}
From the previous lines, it is enough to show for all $3\leq k \leq r$, the following term is at most one:
\begin{equation}\label{kin1}
\frac{1-p^k}{k(1-p)(p\,(2-p))^{k-1}}.
\end{equation}
The difference of (\ref{kin1}) applied at $k+1$ and $k$ is positively proportional to
\begin{equation}\label{krat}
k\left(\sum_{i=0}^kp^i\right)-(k+1)p\,(2-p)\left(\sum_{i=0}^{k-1}p^i\right).
\end{equation}
We will show that this difference is at most zero which implies (\ref{kin1})
is decreasing in $k$ so that it is enough to show the lemma holds in the case where $k=3$.
Notice that $\frac{(k-1)(1-p)}{p} < 1$ implies $k < 1/(1-p)$, so that 
\begin{align}
k\left(\sum_{i=0}^kp^i\right)&-(k+1)p\,(2-p)\left(\sum_{i=0}^{k-1}p^i\right) \notag \\
&=\left(\sum_{i=0}^{k-1}p^i\right)(k(1-p)^2-p\,(2-p))+kp^k \notag \\
&<\left(\sum_{i=0}^{k-1}p^i\right)(1-3p+p^2)+kp^k. \label{ratk}
\end{align}

The small $k$ condition for $k\geq3$ implies in particular that $2/3<p<1$ so that $1-3p+p^2<0$ which, starting from (\ref{ratk}), yields
\begin{align}
k\left(\sum_{i=0}^kp^i\right)&-(k+1)p\,(2-p)\left(\sum_{i=0}^{k-1}p^i\right)\notag \\
&<(1-3p+p^2)+(1-3p+p^2)(k-1)p^{k-1}+kp^k \notag\\
&=(1-3p+p^2)+p^{k-1}(k(1-p)^2-(1-3p+p^2)) \notag\\
&\leq 1-3p+p^2+p^k(2-p) \leq 1-3p+p^2+2p^3-p^4. \label{rak23}
\end{align}
The penultimate inequality follows from the fact noted above that $k < 1/(1-p)$, and the final inequality since $k\geq3$.  From this point it is
a straightforward calculus exercise to show the final term in (\ref{rak23}) is negative for $2/3<p<1$.

For the remaining term, exchangeability implies 
\begin{align}
\mathbb{E}[c_k\mathbb{P}_k(W'=W-1|\{X_i\})]=\mathbb{E}[c_k\mathbb{P}_k(W'=W+1|\{X_i\})], \label{pmeq}
\end{align}
which proves the lemma.
% Finally, for $Y$ negative binomial with parameters $p$ and $k$, the following equalities
% (which are straightforward from definitions) prove the lemma.
% \begin{align}
% c_k\mathbb{E}[\mathbb{P}_k&(W'=W-1|\{X_i\})] 
% = c_kp^k\mathbb{E}\left[\binom{k+Y-2}{k-1}(1-p)^{Y-1}\right] \notag \\
% &=c_kp^k\sum_{i\geq 1} (1-p)^{i-1}\binom{k+i-2}{k-1}(1-p)^i\binom{k+i-1}{k-1}p^k \label{pmeq}\\
% &=c_k\mathbb{E}[\mathbb{P}_k(W'=W+1|\{X_i\})]. \notag
% \end{align}
\end{proof}

In order to prove a result analogous to Lemma \ref{NBords} for the case of $k\geq 1/(1-p)$, the value of $p$ will be restricted.
Ideally, the values of $p$ under consideration should coincide with the values of $p$ where the Poisson distribution is
a good approximation to the negative binomial distribution.  First note that $Var(W)=\lambda+\lambda^2/r$ so that it is not
unreasonable to assume that $\lambda^2 \leq r$. We will instead use a stronger restriction; for the remainder of the section
take $3\lambda^2e^{\lambda} \leq r$.  This may seem like a demanding constraint, but from (\ref{NBapp}) it is clear that
in order for the negative binomial distribution to resemble a Poisson distribution, $e^\lambda$ should be close to $(1+\lambda/r)^r$.  
The next lemma shows that the assumption on $r$ is not unreasonable in lieu of the previous statement; it can be proved
by standard analysis using the Taylor expansion of the appropriate functions.

\begin{lemma}
For $r > \lambda^2 > 0$,
\[\left(\frac{e^\lambda \lambda^2}{7r}\right) \leq e^\lambda-\left(1+\frac{\lambda}{r}\right)^r \leq \left(\frac{e^\lambda \lambda^2}{2r}\right).\]
\end{lemma}

% \begin{proof}
% For the upper bound, we have
% \begin{align*}
% e^\lambda-\left(1+\frac{\lambda}{r}\right)^r &=e^\lambda-\exp\left\{r\log\left(1+\frac{\lambda}{r}\right)\right\} \\
% 					&=e^\lambda\left(1-\exp\left\{\lambda\sum_{i\geq 1}\left(\frac{-\lambda}{r}\right)^i\left(\frac{1}{i+1}\right)\right\}\right) \\
% 					&\leq e^\lambda\left(1-e^{-\lambda^2/(2r)}\right) \leq \frac{e^\lambda \lambda^2}{2r}. \\
% \end{align*}
% Here we use the Taylor expansion of $\log(1+x)$ for $|x| < 1$ (note that since $r\geq1$, $r > \lambda^2 > 0$ implies $r > \lambda > 0$),
% and the inequality $1-e^{-x} \leq x$.

% For the lower bound, continuing from the final equality above,
% \begin{align*}
% e^\lambda-\left(1+\frac{\lambda}{r}\right)^r 
% 						&\geq e^\lambda\left(1-\exp\left\{\frac{-\lambda^2}{r}\left(\frac{1}{2}-\frac{\lambda}{3r}\right) \right\}\right) \\
% 						&\geq e^\lambda\left(1-e^{-\lambda^2/(6r)}\right). 
% \end{align*}
% By truncating the Taylor expansion of $e^{-x}$ appropriately, we obtain
% \begin{align*}
% e^\lambda-\left(1+\frac{\lambda}{r}\right)^r 
% 						&\geq e^\lambda \left(\frac{\lambda^2}{6r}\left(1-\frac{\lambda^2}{12r}\right)\right) \geq \frac{e^\lambda \lambda^2}{7r}.\\
% \end{align*}
% \end{proof}

\begin{remark}
It has been shown \cite{bp99} that $\norm{W-Poi_\lambda}_{TV}\leq \lambda/r$, which implies that for some values of $p$ and $r$
the restriction $3\lambda^2e^\lambda \leq r$ is an overly demanding constraint. It is an interesting problem to consider
what is the minimum constraint that will yield results analogous to Section \ref{binpoisec} and how it relates to
the proximity of the negative binomial to the Poisson distribution.
\end{remark}

\begin{lemma}\label{NBord}
For $3\lambda^2e^\lambda \leq r$ and $k\geq 1/(1-p)$,
both of $\mathbb{E}[c_k\mathbb{P}_k(W'=W+1|\{X_i\})]$ and $\mathbb{E}[c_k\mathbb{P}_k(W'=W-1|\{X_i\})]$ are bounded above by
\[\frac{r(1-p)}{2-p}.\]
\end{lemma}
\begin{proof}
Let $Y$ be a random variable distributed as negative binomial with parameters $p$ and $k$. Then continuing from (\ref{nbb}), for $k \geq2$,
\begin{align*}
c_k\mathbb{E}[\mathbb{P}_k&(W'=W+1|\{X_i\})] \\
%&= c_k(1-p)p^k\mathbb{E}\left[\binom{k+Y}{k-1}(1-p)^{Y}\right]\\
%&=c_k(1-p)p^k\frac{\sum_{l=0}^{k-1}\binom{k}{l+1}\binom{k+l-1}{l}\left(\frac{(1-p)^2}{p(2-p)}\right)^l}{(2-p)^k}\\
&=\frac{r(1-p)}{2-p}\left(\frac{(k-1)(1-p)}{k(2-p)^{k-1}}\right)\sum_{l=0}^{k-1}\binom{k}{l+1}\left(\frac{1-p}{p(2-p)}\right)^l\frac{\binom{k+l-1}{k-1}
\left(1-p\right)^{l}}{\binom{k+a_k-2}{k-2}(1-p)^{a_k}}. \\
\end{align*}
Using the definition of $a_k$, Lemma \ref{ak}, and an argument similar to the use of the constant $b$ in the proof of Lemma
\ref{NB+} (using the fact that $k\geq 1/(1-p)$), we obtain an upper bound of $1/p$ on the appropriate fraction in each summand,
which yields the following inequality.
\begin{align*}
c_k\mathbb{E}[\mathbb{P}_k&(W'=W+1|\{X_i\})] \\
&\leq \frac{r(1-p)}{2-p}\left(\frac{(k-1)}{k(2-p)^{k-2}}\right)\sum_{l=0}^{k-1}\binom{k}{l+1}\left(\frac{1-p}{p(2-p)}\right)^{l+1} \\
&=\left(\frac{r(1-p)}{2-p}\right)\left(\frac{(k-1)}{k(2-p)^{k-2}}\right)\left(\left(\frac{(1-p)}{(2-p)p}+1\right)^k-1\right).
\end{align*}
From the previous lines, it is enough to show for all $2\leq k \leq r$,
\begin{equation}\label{kin}
\left(\frac{(k-1)}{k(2-p)^{k-2}}\right)\left(\left(\frac{(1-p)}{(2-p)p}+1\right)^k-1\right)\leq1.
\end{equation}
The ratio of successive terms is equal to
\[\left(\frac{k^2}{k^2-1}\right)
\frac{\left(\left(\frac{1+p-p^2}{p(2-p)^2}\right)\left(\frac{1+p-p^2}{p(2-p)}\right)^k-(2-p)^{-1}\right)}{\left(\frac{1+p-p^2}{p(2-p)}\right)^k-1},\]
which is at least one.  Thus it is enough to show the inequality (\ref{kin}) for $k=r$. Now, substituting
$p=r/(\lambda+r)$ into (\ref{kin}) yields
\begin{align}
\left(\frac{r-1}{r}\right)&\left(1-\frac{\lambda}{2\lambda+r}\right)^{r-2}\left(\left(1+\frac{\lambda(\lambda+r)}{r(2\lambda+r)}\right)^r-1\right)\notag \\
&\leq e^{-\lambda}\left(\frac{2\lambda+r}{\lambda+r}\right)^{2\lambda+2}\left(e^{\lambda}-1\right). \label{abc}
\end{align}
For the inequality we use the fact that $(1+x/n)^n \leq (1+x/(n+1))^{n+1} \leq e^x$ if $n+x$ is positive and $n \geq 1$.
The term (\ref{abc}) is clearly decreasing in $r$; by using the restriction on the value of $r$ and then the inequality $\log(1+x)\leq x$, we have
\begin{align}
\left(\frac{r-1}{r}\right)&\left(1-\frac{\lambda}{2\lambda+r}\right)^{r-2}\left(\left(1+\frac{\lambda(\lambda+r)}{r(2\lambda+r)}\right)^r-1\right) \notag \\
&\leq \exp\left\{(2\lambda+2)\log\left(1+\frac{1}{1+3\lambda e^\lambda}\right)\right\}\left(1-e^{-\lambda}\right) \notag \\
&\leq \exp\left\{\frac{2\lambda+2}{1+3\lambda e^\lambda}\right\}\left(1-e^{-\lambda}\right).\label{abc1}
\end{align}
Taking the natural logarithm of (\ref{abc1}), we have
\begin{align}
\frac{2\lambda+2}{1+3\lambda e^\lambda}+\log\left(1-e^{-\lambda}\right)&=\frac{2\lambda+2}{1+3\lambda e^\lambda} -\sum_{i\geq1}\frac{e^{-i\lambda}}{i}. \label{abc2} 
\end{align}
The final expression is smaller than any partial sum, and it is easy to see by only taking one term in the sum (\ref{abc2}) is negative for
$\lambda \geq 2$, and taking three terms yields the proper inequality for $\lambda\geq 1$.

For the remaining term, the equation (\ref{pmeq}) continues to hold in this case, which proves the lemma.
\end{proof}

\begin{theorem}\label{NBinc}
For the values of $c_k$ defined previously, and the set of $k$ where either $3\lambda^2e^\lambda \leq r$ and $k\geq 1/(1-p)$ or $k<1/(1-p)$, the error term from Theorem \ref{cdm05} is minimized for $k=1$.
\end{theorem}

\begin{proof}
This follows directly from Lemmas \ref{NBords} and \ref{NBord} and the easily verified fact
\[\mathbb{E}[c_1\mathbb{P}_1(W'=W+1|\{X_i\})]=\mathbb{E}[c_1\mathbb{P}_1(W'=W-1|\{X_i\})]=\frac{r(1-p)}{2-p}.\]
\end{proof}

\section{Acknowledgements}
The author thanks Jason Fulman for the suggestion to write on this topic and an anonymous referee for detailed comments which greatly
improved this work.

\def\polhk#1{\setbox0=\hbox{#1}{\ooalign{\hidewidth
  \lower1.5ex\hbox{`}\hidewidth\crcr\unhbox0}}}

\end{document}